\documentclass{article}
\usepackage{amsmath,amsfonts,amsthm,amscd,amssymb}
\usepackage{pstricks}
\usepackage{caption}
\usepackage{graphicx}
\usepackage{tabulary}
\usepackage{tikz-cd}
\newtheorem{thm}{Theorem}[section]

\newtheorem{cor}[thm]{Corollary}
\newtheorem{lem}[thm]{Lemma}

\newtheorem{example}[thm]{Example}

\begin{document}

\author{Mark Herman and Jonathan Pakianathan}
\title{On a canonical construction of tesselated surfaces via finite group theory, Part II.
}
\maketitle

\begin{abstract}
This paper is the second part of a two-part study of an elementary functorial construction of tesselated surfaces from finite groups. This elementary construction was discussed in the first part and  generally results in a large collection of tesselated surfaces per group, for example when the construction is applied to $\Sigma_6$ it yields $4477$ tesselated surfaces of 27 distinct genus and even more varieties of tesselation cell structure. These tesselations are face and edge transitive and consist of closed cell structures.

In this paper, we continue to study the distribution of these surfaces in various groups and some interesting resulting tesselations with the aid of computer computations. We also show that extensions of groups result in branched coverings between the component  surfaces in their decompositions. Finally we exploit functoriality to obtain interesting faithful, orientation preserving actions of subquotients of these groups and their automorphism groups on these surfaces and in the corresponding mapping class groups.


\end{abstract}

\tableofcontents

\section{Introduction}

In part I, (see \cite{HPY}) we studied a construction which takes a finite nonabelian group $G$ and constructs a $2$-dimensional simplicial complex $X(G)$ in a canonical manner from it. After a canonical resolution of singularities, we obtained an associated complex $Y(G)$ which is a disjoint union of finitely many compact, connected, oriented, triangulated 2-manifolds (we'll call these Riemann surfaces in this paper for brevity at the expense of slight abuse of notation.) The manifolds in $Y(G)$ hence form a natural set of invariants for the group. More specifically the number of times the surface of genus $g$ occurs in $Y(G)$ is a natural invariant of the group which we study.

Additionally, the Riemann surfaces arising via this functorial construction were shown in \cite{HPY} to be equipped with interesting closed-cell structures whose faces are $n$-gons for fixed $n$. These component polytopes have a face and edge transitive automorphism group and at most 2 orbits of vertices (and hence two types of vertex valency). For example, it was shown in \cite{HPY} that when $G$ is the extraspecial $p$-group of order $p^3$ where $p$ is an odd prime, $Y(G)$ yields a regular tesselation of the surface of genus $g=\frac{p(p-3)}{2}+1$ as the union of $p$ $2p$-gons.

Constructions that yield such tesselations on Riemann surfaces are not new. There are classical constructions using triangle or Fuchsian groups,
hyperbolic geometric tesselations of the Poincare disk, Cayley graphs or fundamental group arguments to construct similar objects. Often a set of generators or specific presentation is chosen - the construction in \cite{HPY} is an elementary variant, constructed in an elementary, intrinsic and functorial way from any nonabelian group.

As $X(G)$ and $Y(G)$ are determined functorially from $G$, one
can show that any automorphism of $G$ determines an orientation preserving, simplicial automorphism of $X(G)$. These simplicial automorphisms can permute the surface ``components" of $Y(G)$ but must preserve the genus and indeed the tessellation type. Thus, one can use this construction to find faithful actions of subquotients of $Aut(G)$
on various  surfaces and hence obtain embeddings of these subquotients into the group of orientation preserving diffeomorphisms of $X_g$, the surface of genus $g$.
We explore some examples of orientation preserving group actions afforded by this construction in this paper and some corresponding embeddings into mapping class groups. We also discuss some examples pertaining to the strong symmetric genus of the group, i.e. the smallest genus orientable closed surface for which the group acts on faithfully via orientation preserving diffeomorphisms. See Section~\ref{sec: mapping class} for details.

We attempt to quantify the number and type of surface components that occur in the decomposition $Y(G)$ for a given group $G$. For example when $G$ is the alternating group on $7$ letters, the construction $Y(G)$ discussed in this paper results in $16813$ surface components of $58$ distinct genus, with even more
distinct cell-structures. Tables summarizing
the various polytopes that occur in $Y(G)$ for various groups $G$ can be found in the appendix to this paper.

We also show that extensions of groups yield branched coverings between their constituent
surfaces in their decompositions.

\section{Extensions of Groups and Branched Covers.}
\label{sec: branchedcovers}

\subsection{General Extensions}

We will first consider what happens to the surface decompositions of groups in the construction $Y(G)$ under extension of groups. Please consult the first section of \cite{HPY} for the basic information on the constructions $X(G), Y(G)$ and the notation used in these constructions - we will use this notation throughout this paper.
Let
$$
1 \to N \to \Gamma \overset{\pi}{\to} G \to 1
$$
be an extension. Note that if $\hat{x}, \hat{y}$ in $\Gamma$ do not commute but their images $x,y$ in $G$ commute, then the surface in $Y(\Gamma)$ corresponding
to the edge $[(\hat{x},1),(\hat{y},1)]$ has no corresponding surface in $Y(G)$. However what we will see is that if the images in $G$ do not commute, then there is a nice
correspondence between the component corresponding to the edge $[(\hat{x},1),(\hat{y},1)]$ in $Y(\Gamma)$ and the component corresponding
to the edge $[(x,1),(y,1)]$ in $Y(G)$ given by a branched covering.

Thus let $x,y$ be a non-commuting pair of elements in $G$, and let $S$ be the component of $\hat{Y}(G)$ that is determined by the edge
$[(x,1),(y,1)]$. Fix one choice of lift $\hat{x}$ of $x$ and $\hat{y}$ of $y$ in $\Gamma$. Then the set of all lifts of $x$ to $\Gamma$ is the
set $\{\hat{x}n | n \in N \}$ of cardinality $|N|$ with a similar statement for $y$. Thus there are a total of $|N|^2$ lifts of the edge $[(x,1),(y,1)]$
to $Y(\Gamma)$ consisting of edges of the form $[(\hat{x}n_1,1), (\hat{y}n_2,1)]$ for $n_1, n_2 \in N$. Let $T_{n_1,n_2}$ be the component
determined by the edge $[(\hat{x}n_1,1), (\hat{y}n_2,1)]$ in $Y(\Gamma)$. As $n_1, n_2$ ranges over $N$, we get $|N|^2$ of these components
but they need not be distinct components, as more than one lift edge can live in the same component.
Let $T_1, \dots, T_k$ be a list of distinct components arising this way where the component $T_i$ contains $m_i$ of the lift edges.
Thus $\sum_{i=1}^k m_i = |N|^2$. We will now focus on a single lift component $T$ with $m$ lift edges.

Picking a lift edge in $T$, we get a reference triangle $[(\hat{x}n_1, 1) (\hat{y}n_2,1), (\hat{x}n_1\hat{y}n_2, 2)]$ in $T$ which maps onto the triangle
$[(x,1),(y,1), (xy,2)]$ in $S$ under a simplicial map induced by $\pi$. A simple induction using the fact that an edge in any component determines two triangles adjacent to the edge via
algebraic relations preserved by the quotient homomorphism $\pi$, shows that $\pi$ induces a non-degenerate (faces map to faces) simplicial map from
$T$ onto $S$.

Now suppose $S$ consists of $n$-gon faces in the $2$-cell structure. Consider the $n$-gon sheet containing the edge
$[(x,1),(y,1)]$ in $S$ and a polygonal face lying above it in $T$ containing one of the lifts of the reference edge. Starting at the triangle
$[(\hat{x}n_1,1), (\hat{y}n_2,1),(\alpha=\hat{x}n_1\hat{y}n_2,2)]$ and proceeding around the vertex $(\alpha,2)$ we find after moving by $n$-triangles along
the sheet, we arrive at a triangle lying above the triangle $[(x,1),(y,1),(xy,2)]$ which need not be the starting triangle. Like a spiral staircase, we have arrived at a
location directly above where we started. However by finiteness, there must be a positive integer $\ell$ such that after moving $n\ell$-triangles along
the sheet we arrive back at the starting triangle. Thus the face type of the component $T$ is $n\ell$, a multiple of the face type of the component $S$.
Note that this means that on this polygonal face of $T$, aside from at the midpoint, the map $\pi$ is an $\ell$ to $1$ map. Thus
restricted to this face, $\pi$ induces a $\ell$ to $1$ branched cover over the corresponding polygonal face of $S$, branched only over the midpoint
which has only one preimage (ramification $\ell$). By face transitivity of the cell structures, this is true for all faces of $T$ for a fixed value of $\ell$ independent
of the face. Also note that this face will contain exactly $\ell$ lifts of the reference edge in $S$ as our cell-structures are closed (no self-identifications on rims).
However different faces in $T$ covering a given face in $S$ can have intersecting rims - though their interiors are disjoint. Thus if there are $t$ faces covering
a given face in $S$, this number is independent of the face chosen by face transitivity and the map $\pi$ is a $t\ell$ branched covering map over the interior
of a face of $S$, branched over the midpoint which has only $t$ lifts (ramification index $\ell$). To complete the picture we have to deal with what is happening on the rim of these faces.

First note that two distinct faces in $T$ lying above a given face in $S$ cannot have rims which intersect along an edge. This is because if we had two adjacent faces in $T$ intersecting say in the edge $[(w,1), (z,1)]$ mapping to the same face in $S$, this means
that their midpoints $(wz,2)$ and $(zw,2)$ map to the same element in $S$ which means that the images of $w$ and $z$ commute in $G$. This does not
happen as mentioned earlier as the map $\pi: T \to S$ is well-defined and is a non-degenerate simplicial map. Thus any intersections between two of these faces along the rim happens only at vertices. Thus the $\ell$ lifts of the reference edge within each of the $t$ lift faces are all distinct and must be the full set $m$ of lifts of the
reference edge in that component, i.e., $t\ell=m$.

We hence can consider the sheet centered about $(x,1)$ in $S$ under the triangulation structure of $Y(G)$ which consists of $2\lambda_1$ triangles, where $\lambda_1$ is the valency of $(x,1)$ in the corresponding cell structure (so $\lambda_1$ $2$-cells touch $(x,1)$). Lifting this sheet to $T$ one again
gets a spiral staircase picture which shows that if $\hat{\lambda}_1$ is the valency of $(\hat{x},1)$ in $T$ under the cell-structure, then
$\hat{\lambda}_1=\ell_x \lambda_1$ and $\pi$ restricted to the sheet about $(\hat{x},1)$ in $T$ is a $\ell_x$ to $1$ branched cover over the corresponding
sheet about $(x,1)$ in $S$, branched only over $(x,1)$ with ramification index $\ell_x$. Similar arguments work over $(y,1)$ though with potentially
a different ramification index $\ell_y$. As our cell-structures have at most two orbits of vertices represented by $(x,1)$ and $(y,1)$, we conclude that
the map $\pi: T \to S$ is indeed a branched covering map, branched at most over type 1 and type 2 vertices in the $Y(G)$ structure of the component $S$
with at most $3$ types of ramification index, $\ell, \ell_x$ and $\ell_y$. Though sheets about distinct type 1-vertices do not have to have disjoint interiors
in general, the sheets about the lifts of $(x,1)$ in $T$ must all have disjoint interiors as if they didn't, there would be an edge of the form
$[(xn_1,1), (xn_2,1)]$ in $T$ which maps to a commuting pair in $S$, contradicting the earlier argument that $\pi$ induces a well-defined map to $S$ under our assumptions. Thus if there are $t_x$ sheets in $T$ lying above the sheet about $(x,1)$ in $S$, we conclude $t_x\ell_x$ is the generic $\pi$-preimage size.
Similar comments hold for $y$ in place of $x$.
Comparing with earlier expressions for the total covering number we have
$$
t_x \ell_x = t_y \ell_y = t\ell=m
$$
is the generic preimage size of the branched cover $\pi: T \to S$. As there are $m$ lifts of the reference edge of $S$ to $T$, considering the points
above the interior of the edge, we see that the total covering number is equal to $m$ and hence that $\ell_x, \ell_y$ and $\ell$ divide $m$.

We summarize what we have shown in the following theorem:

\begin{thm}
\label{thm: branchedcover}
Let $1 \to N \to \Gamma \overset{\pi}{\to} G \to 1$ be an extension of finite groups and suppose $x,y$ do not commute in $G$. Let $S$ be the component determined
by $[(x,1),(y,1)]$ in $\hat{Y}(G)$. Then the $|N|^2$-lifts of the edge $[(x,1),(y,1)]$ to $\hat{Y}(\Gamma)$ determine $k$ distinct components
$T_1, \dots T_k$ where $T_j$ contains $m_j$ of these edge lifts. Then: \\
(1) $\sum_{j=1}^k m_j = |N|^2$. \\
(2) If $T$ is a component containing $m$ lifts of the reference edge then $\pi$ induces a $m$-fold branched cover $T \to S$ which is branched only over
type 1 and 2 vertices in $S$ with at most three ramification indices denoted $\ell$ (over the type 2 vertices), $\ell_x$ and $\ell_y$ (over the at most two orbits of type 1 vertices) which are divisors of $m$. Furthermore if the Schl\"afli symbol
of $S$ is $\{n, \lambda_1\text{-}\lambda_2\}$ then the Schl\"afli symbol of $T$ is $\{n\ell, \lambda_1\ell_x\text{-}\lambda_2\ell_y\}$.
\end{thm}

Note that all the ramification that occurs can be deduced from the change in Schl\"afli symbols as we lift the component $S$ to the component $T$.
The value of $t=\frac{m}{\ell}$ can be computed from face counts as $F'=tF$ where $F'$ is the number of faces in $T$ and $F$ is the number of faces in $S$. Further note that edge and vertex counts are related by $E'=mE$, $V_1'=\frac{m}{\ell_x}V_1, V_2'=\frac{m}{\ell_y}V_2$.

\subsection{Central Extensions}

A little more can be said in the case of central extensions
$$
1 \to C \to \Gamma \overset{\pi}{\to} G \to 1.
$$
In the case of a component $S$ in $G$ with two orbits of vertices under conjugation, we find the component lifts $T_{c_i,c_j}, c_i,c_j \in C$ have the same property.
Fix $c, c' \in C$. We can define a map $\Psi_{c,c'}$ on vertices which takes any type 1 vertex in the first orbit and multiplies it by $c$ and any
type 1 vertex in the second orbit and multiplies it by $c'$ and finally takes any type 2 vertex and multiplies it by $cc'$. Notice this map takes the
vertices $[(\alpha,1),(\beta,1),(\alpha\beta,2)]$ to $[(\alpha c,1), (\beta c',1), (\alpha\beta cc',2)]$ and hence induces a well-defined simplicial map
of the union of the $T_{c_i,c_j}$ to itself which we will call the ``central twist" $\Psi_{c,c'}$. In fact $C \times C$ acts on the union of $T_{c_i,c_j}$ via
these maps transitively permuting these components amongst themselves. Thus the genus and cell type of each of these lift components is the same
and hence each of them contain the same number $m$ of lifts of the reference edge. If there are $k$ distinct lifts then $mk=|C|^2$.

Furthermore when considering the ramification picture over a $n$-gon $2$-cell of $S$, one finds upon traversing along a lift sheet in $T$ for $n$ triangle steps
starting at $[(\hat{x},1),(\hat{y},1),(\hat{x}\hat{y},2)]$ one returns to a triangle of the form $[(\hat{x}c,1), (\hat{y}c^{-1},1),(\hat{x}\hat{y},2)]$ where $c \in C$
is a ``monodromy element". A simple computation using that movement along the sheet rim is accomplished by conjugations, then finds that the face type in $T$, which is a $n\ell$-gon, will have $\ell$, the ramification index
equal to the order of this monodromy element. Thus in particular $\ell$ divides the exponent of the group $C$. Similar considerations work for
$\ell_x$ and $\ell_y$.

On the other hand if $S$ has a single conjugation orbit of vertices, so can its lifts $T_{c_i,c_j}$ and so we can no longer distinguish into two kinds of type 1 vertices.
Thus only the central twists of the form $\Psi_{c,c}$ are still well-defined. These yield an action of $C$ on the lifts $T_{c_i,c_j}$ which is no longer transitive but breaks up into $|C|$ many groupings of components each containing $|C|$ lifts of the reference edge. Within each grouping, the components must have the same genus, cell-structure and branching structure over $S$.  If $k_i$ denotes the number of components in one of the groupings and $m_i$ the common covering number
then $k_im_i=|C|$. The arguments about monodromy still work so $\ell$ and $\ell_x$ still divide the exponent of $C$.

We record these observations in the next theorem:

\begin{thm}
\label{thm: central branched cover}
Let $1 \to C \to \Gamma \to G \to 1$ be a central extension of finite groups. Let $S$ be a component of $\hat{Y}(G)$, then in addition to the results of
Theorem~\ref{thm: branchedcover} one has: \\
(1) If $S$ has two orbits of vertices under conjugation then all the components lying above $S$ in $\hat{Y}(\Gamma)$ have the same genus and cell type
and branching structure over $S$. If there are $k$ distinct such components and each yields a $m$-fold branched cover over $S$ then $mk=|C|^2$.
The ramification indices $\ell, \ell_x$ and $\ell_y$ divide the exponent of $C$. \\
(2) If $S$ has a single orbit of vertices under conjugation then the components lying above $S$ break up into $|C|$ groupings. Within each grouping
all components have the same genus, cell-structure and branching structure. If $m_i$ and $k_i$ denote the covering number and number of components in a
given grouping then we have $m_ik_i=|C|$. The ramification indices $\ell, \ell_x$ still divide the exponent of $C$.
\end{thm}


\subsection{Example: Extension of $\Sigma_3$ to $\Sigma_4$}

Here we present the data in theorem \ref{thm: branchedcover} for
$$
1 \to N \to \Sigma_4 \overset{\pi}{\to} \Sigma_3 \to 1,
$$
where $N=\{(), (12)(34), (13)(24), (14)(23)\} \cong \mathbb{Z}_2 \times \mathbb{Z}_2$.
From tables \ref{table:Sigma3} and \ref{table:Sigma4}, we see there are two (genus 0) components of $\hat{Y}(\Sigma_3)$, which we refer to by their face count and Schl\"afli symbol as $\{3, 2\}^2$ and $\{4,2\text{-}3\}^3$ (the superscript refers to the face count). We similarly refer to the components of $\hat{Y}(\Sigma_4)$.

The $\{3, 2\}^2$ component of $\hat{Y}(\Sigma_3)$ is covered by six components of $\hat{Y}(\Sigma_4)$:
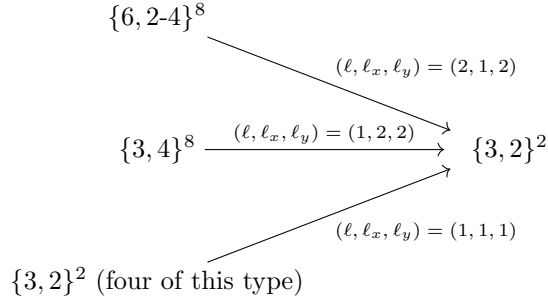
\begin{center}
\begin{tikzcd}[column sep=5em]
\{6,2\text{-}4\}^8 \arrow{rd}{(\ell,\ell_x,\ell_y)=(2,1,2)} & \\[.5cm]
\{3,4\}^8 \arrow{r}{(\ell,\ell_x,\ell_y)=(1,2,2)} & \ \ \{3,2\}^2 \\[.5cm]
\{3,2\}^2 \text{ (four of this type)} \arrow{ru}[swap]{(\ell,\ell_x,\ell_y)=(1,1,1)} &
\end{tikzcd}
\end{center}
On the left are the components in $\hat{Y}(\Sigma_4)$ which map under $\pi$ to the component of $\hat{Y}(\Sigma_3)$ on the right. The ramification indices $(\ell,\ell_x,\ell_y)$ are those referenced in theorem~\ref{thm: branchedcover} and along with face counts of the components, determine all the relevant data. For example, there are four $(1,1,1)$ coverings which are simply one-to-one (no branching), while the $(1,2,2)$ covering is branched over all type 1 vertices (but not type 2). The face counts tell us, for instance, that for the $(2,1,2)$ covering, there are $t=4$ faces in the $\hat{Y}(\Sigma_4)$ component over each face in the $\hat{Y}(\Sigma_3)$ component (since an 8 face component maps to a 2 face component). As expected, there are $|N|^2=16$ total edge lifts of each edge in the $\hat{Y}(\Sigma_3)$ component. Indeed, the $(2,1,2)$ covering accounts for 8 of them, the $(1,2,2)$ covering accounts for 4, and the $(1,1,1)$ coverings each account for 1.

The $\{4,2\text{-}3\}^3$ component of $\hat{Y}(\Sigma_3)$ is covered by seven components of $\hat{Y}(\Sigma_4)$:

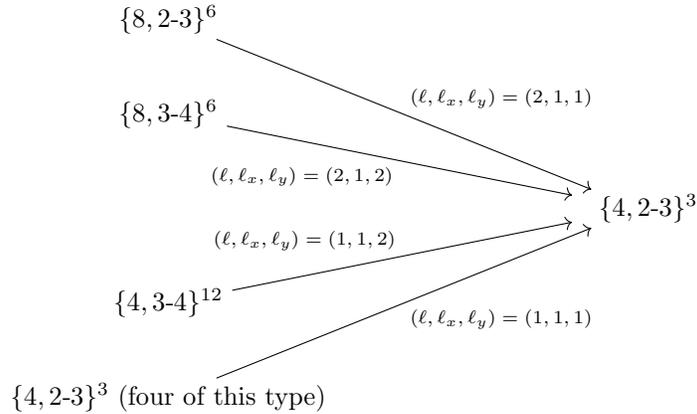
\begin{center}
\begin{tikzcd}[column sep=9em]
\{8,2\text{-}3\}^6 \arrow{rdd}{(\ell,\ell_x,\ell_y)=(2,1,1)} & \\
\{8,3\text{-}4\}^6 \arrow{rd}[swap]{(\ell,\ell_x,\ell_y)=(2,1,2)} &  \\
& \ \ \{4,2\text{-}3\}^3 \\
\{4,3\text{-}4\}^{12} \arrow{ru}{(\ell,\ell_x,\ell_y)=(1,1,2)} &  \\
\{4,2\text{-}3\}^3 \text{ (four of this type)} \arrow{ruu}[swap]{(\ell,\ell_x,\ell_y)=(1,1,1)} &
\end{tikzcd}
\end{center}

Note that for all other components of $\hat{Y}(\Sigma_4)$ there is no corresponding surface in $\hat{Y}(\Sigma_3)$. For these components of $\hat{Y}(\Sigma_4)$, each edge $[(\hat{x},1),(\hat{y},1)]$ ``maps to a commuting pair'' under $\pi$.

\subsection{Example: Central Extension of $PSL(2,\mathbb{F}_3)$ to $SL(2,\mathbb{F}_3)$}

Here we present the data in theorem \ref{thm: central branched cover} for the central extension
$$
1 \to C \to SL(2,\mathbb{F}_3) \overset{\pi}{\to} PSL(2,\mathbb{F}_3) \to 1.
$$
From tables \ref{table:SL2F3} and \ref{table:Alt4}, we see there are five (genus 0) components of $\hat{Y}(PSL(2,\mathbb{F}_3))$, which we refer to by their face count and Schl\"afli symbol as in the previous example. Here, the only branching occurs over the two $\{3,3\}^4$ components of $\hat{Y}(PSL(2,\mathbb{F}_3))$ given below:

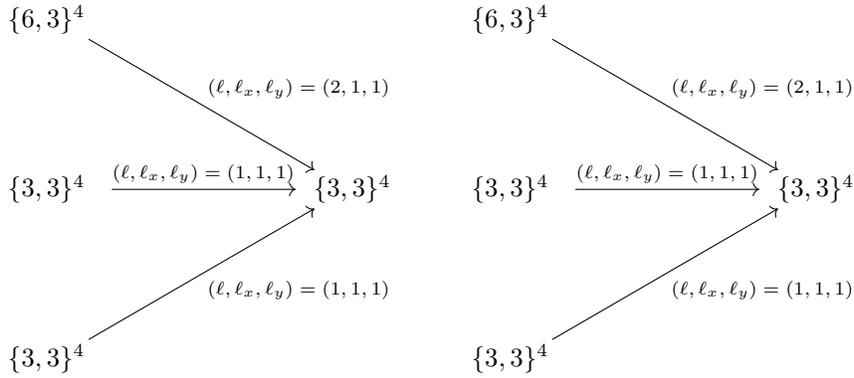
\begin{center}
\begin{tikzcd}[column sep=7em]
\{6,3\}^4\ \ \arrow{rd}{(\ell,\ell_x,\ell_y)=(2,1,1)} & \\[1cm]
\{3,3\}^4\ \ \arrow{r}{(\ell,\ell_x,\ell_y)=(1,1,1)} & \ \{3,3\}^4 \\[1cm]
\{3,3\}^4\ \ \arrow{ru}[swap]{(\ell,\ell_x,\ell_y)=(1,1,1)} &
\end{tikzcd}
\hskip .5cm
\begin{tikzcd}[column sep=7em]
\{6,3\}^4 \ \ \arrow{rd}{(\ell,\ell_x,\ell_y)=(2,1,1)} & \\[1cm]
\{3,3\}^4 \ \ \arrow{r}{(\ell,\ell_x,\ell_y)=(1,1,1)} & \ \{3,3\}^4 \\[1cm]
\{3,3\}^4 \ \ \arrow{ru}[swap]{(\ell,\ell_x,\ell_y)=(1,1,1)} &
\end{tikzcd}
\end{center}

The remaining three components of $\hat{Y}(PSL(2,\mathbb{F}_3))$ are covered in a one-to-one fashion, i.e. $(\ell,\ell_x,\ell_y)=(1,1,1)$, by isomorphic components in $\hat{Y}(SL(2,\mathbb{F}_3))$.

Note that for all other components of $\hat{Y}(SL(2,\mathbb{F}_3))$ there is no corresponding surface in $\hat{Y}(PSL(2,\mathbb{F}_3))$. For these components of $\hat{Y}(SL(2,\mathbb{F}_3))$, each edge $[(\hat{x},1),(\hat{y},1)]$ ``maps to a commuting pair'' under $\pi$.

\section{Application: Group Actions on Surfaces}
\label{sec: mapping class}

By functoriality, $Aut(G)$ acts on $Y(G)$ via orientation preserving simplicial automorphisms. Using standard equivalences between the PL and smooth categories in dimension $2$, we can view this action as one through orientation preserving diffeomorphisms on the corresponding smoothened surfaces constituting $Y(G)$. If $S$ is one of the surface components, then $Aut(G)$ permutes this component amongst others. All components in the same $Aut(G)$-orbit
as the component $S$ have to have the same genus and cell-structure type. Let us denote the component stabilizer as $Aut(G)_S$, thus $Aut(G)_S$ consists
of the elements of $Aut(G)$ which map the component $S$ back to itself. The index of $Aut(G)_S$ in $Aut(G)$ gives the number of components in the component orbit of $S$ and
$Aut(G)_S$ acts on the surface $S$ itself via orientation preserving simplicial automorphisms.

By connectedness and the fact that the action is orientation preserving, the elements of $Aut(G)_S$ that fix a reference edge pointwise are precisely the elements that
act trivially on the whole component. Modding out this subgroup, we obtain a faithful, orientation preserving, action on the component $S$ thru cellular
automorphisms. We summarize these facts in the next theorem:

\begin{thm}
\label{thm: actions}
Let $G$ be a finite nonabelian group and let $S$ be a component of $Y(G)$ of genus $g$ containing a reference edge $[(x,1),(y,1)]$. Then $Aut(G)$ acts on $\hat{Y}(G)$ permuting components and the orbit of
$S$ consists of components with the same genus and cell-structure type. The stabilizer $Aut(G)_S$ consists of the elements of $Aut(G)$ which map $S$ back
to itself and this subgroup acts on $S$ via orientation preserving cell-automorphisms (diffeomorphisms in the corresponding smoothened surface). Then: \\
(1) The index of $Aut(G)_S$ in $Aut(G)$ is no more than the number of components in $\hat{Y}(G)$ which have the same genus and cell-structure as $S$. \\
(2) If $Aut_x$ and $Aut_y$ denote the stabilizers of the vertices $(x,1)$ and $(y,1)$ respectively then $Aut_x \cap Aut_y$ is the kernel of the
action of $Aut(G)_S$ on $S$ and hence is a normal subgroup of $Aut(G)_S$. \\
(3) Let $Q_S$ be defined by the short exact sequence
$$
1 \to Aut_x \cap Aut_y \to Aut(G)_S \to Q_S \to 1
$$
then $Q_S$ is a sub-quotient of $Aut(G)$ and $Q_S$ acts faithfully on the surface of genus $g$ via orientation preserving diffeomorphisms
which preserve the cell-structure of the component $S$. Thus $Q_S$ embeds into $Diff^+(X_g)$ the group of orientation preserving diffeomorphisms
of the closed surface of genus $g$. We will call the group $Q_S$ the sub-quotient of $Aut(G)$ associated to the component $S$. \\
(4) As the component $S$ is edge transitive, the set-stablizers of edges form a conjugacy class. The set-stabilizer of our reference edge in
$Q_S$ will be denoted by $Q_e$. Thus $Q_e$ is the subgroup of elements of $Q_S$ which take the edge $[(x,1),(y,1)]$ back to itself as a set.
Then $|Q_e| \leq 2$. \\
(5) All vertex stabilizers (of type $1$ vertices) are conjugate to $Q_x$ and/or $Q_y$, where $Q_x$ and $Q_y$ denote the stabilizers in $Q_S$ of the vertex $(x,1)$
and $(y,1)$ respectively. $Q_x \cap Q_y = \{ 1 \}$. Let $\lambda_1$ and $\lambda_2$ denote the valencies of the vertices $(x,1)$ and $(y,1)$ then $Q_x$ and $Q_y$ are cyclic groups of order $\lambda_1$ and $\lambda_2$ generated by
$c_x$ and $c_y$ respectively where $c_t$ denotes conjugation by $t$. \\
(6) As the component $S$ is face transitive in the closed cell-structure, the set-stabilizers of faces in the closed cell structure of the component form a conjugacy class. The set-stabilizer in $Q_S$ of a reference face containing
the triangle \\ $[(x,1), (y,1), (xy,2)]$ will be denoted $Q_F$. $Q_F$ is also the vertex stabilizer of the type $2$ vertex $(xy,2)$ in the center of that face. If the faces of $S$ in the cell-structure are $n$-gons, then $Q_F$ is a cyclic group of order either $n$ or $\frac{n}{2}$. \\
(7) The group $Q_S$ acts on the surface $S$ with cyclic stabilizer groups. The only points which can have nontrivial stabilizer groups are
vertices of type $1$ or vertices of type $2$ or midpoints of edges joining two type 1 vertices. In all cases $Q_S$ acts freely (with trivial stabilizer groups), preserving orientation, on the
unit-tangent bundle $T_g$ of $X_g$, the surface of genus $g$. $T_g$ is a compact, connected, orientable $3$-manifold.
\end{thm}
\begin{proof}
The initial part of the statement of the theorem and (1) follow from the comments preceding the theorem. Thus we will concentrate on proving (2)-(7). \\
Proof of (2) and (3): First observe that an orientation preserving simplicial automorphism of a triangulated surface that fixes an edge pointwise, must fix
all points in the two triangles adjacent to that edge as it preserves orientation. As the elements of $Aut_x \cap Aut_y$ fix the reference edge pointwise,
a connectedness argument shows that it fixes all triangles in the component pointwise and hence acts trivially in the component. On the other
hand elements not in $Aut_x \cap Aut_y$ move either the vertex $(x,1)$ or vertex $(y,1)$ and hence act nontrivially on the component.
Thus $Aut_x \cap Aut_y$ is exactly the kernel of the action of $Aut(G)_S$ on $S$. (2) and (3) follow. \\
Proof of (4): A simplicial map which takes the reference edge back to itself as a set either fixes the end vertices, in which case it becomes the trivial element
in $Q_S$, or it flips them. If $\tau$ is such an automorphism which takes $x$ to $y$ and vice-versa, then any other element which does the same differs
from $\tau$ by an element which fixes the component pointwise and hences projects to the same thing as $\tau$ does in $Q_S$. Thus $Q_e$ is cyclic
generated by $\bar{\tau}$, the image of $\tau$, and $\bar{\tau}$ clearly has order two. \\
Proof of (5): All type 1 vertices in $S$ are conjugate to a vertex in the reference edge. $Q_x$ is exactly the image of $Aut_x$ in $Q_S$ and similarly
$Q_y$ is the image of $Aut_y$. As $Aut_x \cap Aut_y$ was the kernel of the action, $Q_x \cap Q_y = \{ 1 \}$. Recall we saw that $c_x$, conjugation by $x$,
permuted the $2$-cell faces incident to $(x,1)$ transitively and that there were $\lambda_1$ $2$-cell faces in the orbit where $\lambda_1$ was the valency of $x$.
If $\gamma$ is an automorphism of the component $S$ which takes $(x,1)$ to itself, then $\gamma$ must induce a permutation of the faces incident to
$(x,1)$. Thus the composition of $\gamma$ with a suitable power of $c_x$ will fix the reference edge $[(x,1),(y,1)]$ pointwise and hence induce the trivial
automorphism of the component. Thus in $Q_S$, $\bar{\gamma}$, the image of $\gamma$, lies in the cyclic group generated by $\bar{c}_x$.
Thus $Q_x$ is cyclic and is generated by the image of $c_x$ in $Q_S$ and has order $\lambda_1$. Similar arguments work for $y$ in place of $x$.
This proves (5). \\
Proof of (6): Let $\tau$ denote an automorphism that fixes the reference face setwise. Then the composition of $\tau$ and a suitable power of the conjugation $c_{xy}$ takes this face to itself and the reference edge either to itself or an adjacent edge in the same face. The latter case need only be considered if the
reference edge and the adjacent edge are not $xy$-conjugate. If this latter case exists then there exists an automorphism $\tau'$ which fixes the face
and takes the reference edge to an adjacent edge. The cyclic group generated by $\tau'$ then acts transitively on the edges of the face. Any other automorphism which fixes the face is readily seen to be equal to the composite of a power of $\tau'$ and an automorphism which fixes the reference edge pointwise.
Thus the image of $\tau'$ generates $Q_F$ and $Q_F$ is a cyclic group of order $n$. On the other hand, if there is no automorphism of the component
which takes the reference edge to an adjacent edge, then any automorphism that fixes the face setwise is equal to the composite of a power of $c_{xy}$ and
an automorphism which acts trivially on the component. In this case, $Q_F$ is cyclic generated by $c_{xy}$ and has order $\frac{n}{2}$. \\
Proof of (7): Recall that any element of $Aut(G)_S$ fixing an edge of $S$ pointwise, must act trivially on all of $S$ and hence project to the trivial element
in $Q_S$. Thus in the faithful $Q_S$ action on $S$ with its $2$-cell structure, any nontrivial element that fixes an edge (between two type 1 vertices) setwise must flip the end-vertices. Thus the only points with nontrivial stabilizers on such an edge are the midpoint (stabilizer is at most a cyclic group of order 2) and
the two end vertices (cyclic stabilizers of order $\lambda_1$ and $\lambda_2$). Considering the triangulation structure underlying the $2$-cell structure of the
component, each triangle is of the form $[(a,1),(b,1),(ab,2)]$ and can only be taken to itself by an automorphism that fixes $(ab,2)$. Such an automorphism
either fixes the edge $[(a,1),(b,1)]$ pointwise in which case it induces the trivial automorphism of $S$ or it flips it. However flipping an edge and taking
the triangle back to the same triangle violates orientation preservingness. Thus no nontrivial element of $Q_S$ fixes any triangle as a set and this shows
that $Q_F$, the face stabilizer of the $n$-gon face has no nontrivial element fix anything besides the midpoint of the face $(ab,2)$.
Putting this all together, we see that the only points with nontrivial stabilizers in the $Q_S$ action on $S$ are the type 1 vertices, the midpoints of the $2$-cell faces (type 2 vertices) and possibly the midpoints on the edges (between type $1$ vertices). Each of these stabilizer groups has been shown to be cyclic.
Finally in the corresponding smooth action, it is clear that any nontrivial element of $Q_S$ that fixes a point in $S$, acts nontrivially on the tangent space
of that point. This translates to the corresponding smooth action of $Q_S$ on the unit tangent bundle of $S$ having trivial stabilizers, i.e, being a free action.
The unit tangent bundle over $X_g$ is the total space of an oriented circle bundle over $X_g$ and hence is a compact, connected, oriented, $3$-manifold.
\end{proof}

Theorem~\ref{thm: actions} goes through similarly when the group $G$ is used in place of $Aut(G)$. The group $G$ acts on itself through
conjugations (inner automorphisms). Given a component $S$ in $\hat{Y}(G)$ a subgroup $G_S$ of $G$ will stabilize that component and act on it
through orientation preserving cellular automorphisms. The kernel of this action is $C(x) \cap C(y)$ and thus one obtains a subquotient
$$
1 \to C(x) \cap C(y) \to G_S \to \bar{Q}_S \to 1.
$$
Thus $\bar{Q}_S$ acts faithfully on $S$ through orientation preserving cellular automorphisms.
As the component is edge and face transitive under the conjugation action already, there is only a small difference between this subquotient $\bar{Q}_S$ of $G$
associated to the component $S$ and the subquotient $Q_S$ of $Aut(G)$ associated to the component $S$. Clearly $\bar{Q}_S$ embeds as a
subgroup of $Q_S$ consisting of those automorphisms of the surface induced by inner automorphisms. The (type 1) vertex stabilizers are isomorphic,
i.e., $Q_x = \bar{Q}_x$ and $Q_y = \bar{Q}_y$ while the edge and face stabilizers $\bar{Q}_e$ and $\bar{Q}_F$ are index $\leq 2$ subgroups of
$Q_e$ and $Q_F$ respectively. Note that this means the stabilizer groups of the $\bar{Q}_S$ action on $S$ are cyclic also and that this action also yields a
free orientation preserving action on the unit tangent bundle of $S$.

Recall that the unit tangent bundle $U(X_g)$ of the closed surface of genus $g$, $X_g$, is the total space of an oriented circle bundle over $X_g$.
When $g=0$, $U(X_0)$ is diffeomorphic to $SO(3)$ or equivalently real projective $3$-space $\mathbb{R}P^3$. When $g=1$, the fact that the torus
is parallelizable yields that $U(X_1)$ is diffeomorphic to the $3$-torus $T^3 = S^1 \times S^1 \times S^1=K(\mathbb{Z}^3,1)$. For $g \geq 2$,
$U(X_g)$ is a $K(\pi,1)$ space where $\pi$ is a $3$-dimensional Poincare duality group which fits in a short exact sequence
$$
1 \to \mathbb{Z} \to \pi \to \pi(X_g) \to 1
$$
where $\pi(X_g)$ is the fundamental group of the surface of genus $g$ given by the presentation
$< a_1, b_1, a_2, b_2, \dots, a_g, b_g | [a_1,b_1][a_2,b_2]\dots[a_g,b_g] = 1 >$.

If $Q$ is one of the subquotients given by Theorem~\ref{thm: actions} (or the equivalent when $Aut(G)$ is replaced with $G$) acting freely on $U(X_g)$
then we get a covering space $U(X_g) \to U(X_g)/Q $ where the orbit space $U(X_g)/Q$ is a closed connected oriented $3$-manifold.
For $g \geq 1$, $U(X_g)/Q$ and $U(X_g)$ have no higher homotopy and hence have fundamental groups which are $3$-dimensional
Poincare duality groups and in particular are torsion free (see \cite{Brown}).
Thus by basic covering space theory one gets a short exact sequence
$$
1 \to \pi_1(U(X_g)) \to \pi_1(U(X_g)/Q) \to Q \to 1
$$
which exhibits $Q$ as a quotient of one $3$-dimensional Poincare duality group by another.

Another remark is that the size of a finite group of orientation preserving homeomorphisms of $X_g$, the closed surface of genus $g$, is bounded
in general by $84(g-1)$ when $g \geq 2$. This upper bound can be
achieved for infinitely many $g$ but also cannot be achieved for infinitely many $g$. A Riemann surface with a finite conformal automorphism group achieving this bound is called a Hurwitz surface. This result was proven first by Hurwitz for conformal maps and then generalized to the case of homeomorphisms in \cite{TT}.
For more on these issues and for a list of open questions regarding these surfaces see \cite{F}. Thus for any subquotient $Q$
associated to a component of genus $g \geq 2$ in the construction $\hat{Y}(G)$, we can conclude that $|Q| \leq 84(g-1)$. We record this as a
corollary:

\begin{cor}[Hurwitz bound]
Let $Q$ be a subquotient associated to a component of $\hat{Y}(G)$ of genus $g \geq 2$. Then $|Q| \leq 84(g-1)$.
This upper bound is only achieved for genus $g$ where a Hurwitz surface exists.
\end{cor}

We will next prove that if $Q$ is a subquotient acting on a component $S$ arising in the construction $\hat{Y}(G)$, then the orbit quotient map $S \to S/Q$ is a branched covering over the sphere, branched (or equivalently ramified) over exactly $3$ branch points.

\begin{thm}[Components are branched over the Riemann Sphere]
\label{thm: ram}
Let $G$ be a finite nonabelian group and let $S$ be a component occurring in $\hat{Y}(G)$
with corresponding sub quotient $Q$ and $|E|$ edges. Let a representative edge of $S$ be $[(x,1),(y,1)]$. Then: \\
(1) The quotient map $\pi: S \to S/Q$ is a branched $|Q|$-fold covering map, branched over exactly three points. $S/Q$ is homeomorphic to the $2$-sphere. \\
(2) In the case that the component consists of two $Q$-orbits of vertices, then the three branch points in $S/Q$ consist of the image of $(x,1)$, the image of $(y,1)$
and the image of the center of faces with ramification indices $\lambda_1, \lambda_2$ and $\frac{n}{2}$ respectively (note $\lambda_1=\lambda_2$ is possible in this case also). In this case,
$Q$ is a subgroup of index $\leq 2$ of the orientation preserving automorphisms of the cell structure of $S$ and $|Q|=|E|$. In the non-equivar case
$\lambda_1 \neq \lambda_2$, $Q$ is the whole group of orientation preserving automorphisms of the cell-structure of $S$. \\
(3) In the case that the component consists of a single $Q$-orbit of vertices, then the three branch points in $S/Q$ consist of the image of $(x,1)$, the image
of the midpoints of edges, and the image of the center of faces with ramification indices $\lambda, 2$ and $n$ respectively. In this case, $Q$ is the group of orientation preserving automorphisms of the cell structure of $S$ and $|Q|=2|E|$.
\end{thm}
\begin{proof}
By Theorem~\ref{thm: actions}, outside the finite number of points with nontrivial stabilizer groups, the $Q$-action is free and the map $\pi: S \to S/Q$ is a covering map. Thus it is clear $\pi$ is a $|Q|$-fold branched cover, branched only over the images of orbits with nontrivial stabilizer and $S/Q$ is hence a orientable,
compact, connected surface of genus $\bar{g}$ to be determined.

Furthermore note that there are two $Q$-orbits of vertices in $S$ (in the cell-structure so we are only discussing "type 1" vertices) if and only if
there is no element of $Q$ flipping an edge. This is because if an element of $Q$ flips an edge, as we know the action is edge transitive, it follows it is also
vertex transitive. On the other hand, if $Q$ acts vertex transitively, by composing the element of $Q$ that maps $x$ to $y$ with a suitable power of $y$-conjugations (which permute all the edges incident to $(y,1)$), we will arrive at an element of $Q$ which flips the edge $[(x,1),(y,1)]$.

The rest of the proof, breaks up into two cases, one where the $Q$-action on vertices is transitive and one where it has two orbits. We will do the two orbit case
and leave the easier one orbit case to the reader as the proof is similar.

In the case that the $Q$-action on vertices has two orbits, $Q_e=\{ 1 \}$ as no element of $Q$ can flip an edge. Thus there are exactly $3$ orbits with nontrivial stabilizers, that of $(x,1)$, that of $(y,1)$ and that of the midpoint of $2$-cell faces, for example $(xy,2)$. The corresponding ramification indices (size of stabilizer groups) are $|Q_x|=\lambda_1$, $|Q_y|=\lambda_2$ and
$|Q_F|=\frac{n}{2}$ by Theorem~\ref{thm: actions}. As $Q$ acts faithfully and edge transitively, without edge flips, $|Q|=|E|$ and $Q$ can have at most index $\leq 2$ in the group of orientation preserving
automorphisms of the cell-structure of $S$. (The only possible nontrivial coset of $Q$ being that by an automorphism flipping an edge.)

The Riemann-Hurwitz formula applied to the branched cover $\pi: S \to S/Q$ yields:

\begin{eqnarray*}
2-2g &=& |Q|(2-2\bar{g}) - (|Q|-\frac{|Q|}{\lambda_1}) - (|Q|-\frac{|Q|}{\lambda_2}) - (|Q|-\frac{2|Q|}{n}) \\
&=& |Q|(\frac{1}{\lambda_1} + \frac{1}{\lambda_2} + \frac{2}{n}-1-2\bar{g})
\end{eqnarray*}

On the other hand using $|Q|=|E|$ in Theorem 2.32 of \cite{HPY} yields
$$
2-2g =|Q|(\frac{1}{\lambda_1} + \frac{1}{\lambda_2}+ \frac{2}{n}-1)
$$
Comparing this with the equation above it we conclude $\bar{g}=0$ i.e. $S/Q$ is a sphere. This completes the proof.
\end{proof}

Note as a consequence of Theorem~\ref{thm: ram}, each component $S$ is a branched cover $S \to S/Q$ over the sphere, branched over exactly $3$ points. As the sphere has a unique complex structure (that of the Riemann sphere), and its conformal automorphism group is sharply $3$-transitive, we can assume
those $3$ points are $0, 1$ and $\infty$. Then, away from the finite number of points with nontrivial $Q$-stablizers in $S$, the complex structure of the sphere will lift to a unique complex structure on $S$ such that the deck transformations in $Q$ act via conformal maps. One of course has to discuss the lifting of the complex structure about the ramified points also but as this is not the focus of this paper, we will not do so here. However this shows that the various $Q$-actions on the individual components indeed connect with the Fuchsian group constructions obtained via complex analysis classically for the group $Q$ and that one
can indeed picture these components as genuine Riemann surfaces being acted on by $Q$ via conformal automorphisms.

As another consequence, the subquotient $Q$ associated to a given component $X$ either for the $G$ or $Aut(G)$ action must be a group generated by two elements.
This follows from considering the $Q$-covering over the sphere minus the three branch points which yields a short exact sequence:
$$
1 \to \pi_1(X-\text{ramified points}) \to \pi_1(S^2-\text{3 points}) \to Q \to 1
$$
and the fact that $\pi_1(S^2-\text{3 points})=\pi_1(\mathbb{R}^2-\text{2 points})$ is a free group on two generators.

Let us now look at some examples of the actions constructed via the functor $\hat{Y}(G)$.

\subsection{Genus zero actions}

In \cite{TT} the concept of strong symmetric genus of a finite group was introduced. The strong symmetric genus of a finite group $G$
is the minimal genus $g \geq 0$ such that $G$ acts faithfully via orientation preserving homeomorphisms on the closed surface of genus $g$.
A complete classification of groups with a given strong symmetric genus has been obtained for genus $g \leq 25$ in the papers
\cite{MZ1}, \cite{MZ3}, \cite{FTX}. There exist a finite, nonzero number of finite groups (up to isomorphism) of any given strong symmetric genus $g \geq 2$ (see \cite{MZ2}).

Note if $Q$ is one of the sub quotients corresponding to a component $S$ of genus $g$ obtained via Theorem~\ref{thm: actions},
then the strong symmetric genus of $Q$ is $\leq g$. We will now use our construction to recover all finite groups of strong symmetric genus $0$.
Though this list is classical and well known, we point out that this construction hands these actions to us without any need for inspiration. The group builds
the surface canonically via the functor $Y(G)$, the surface thus constructed comes with labels of all vertices by group elements and we merely have to calculate how the group acts on itself via conjugation to visualize the resulting action on the surface. Thus basically the construction removes any need for special inspiration to find these actions.

\begin{example}[Dihedral Groups]
Let $D_{2n}$ be the dihedral group of order $2n$ where $n \geq 3$ and let $\tau$ be a generator of the cyclic subgroup of "rotations". Then as discussed in Section 3.1 of \cite{HPY}, there is a sphere component
with Schl\"afli symbol $\{n,2\}$ in $\hat{Y}(D_{2n})$ where the north and south pole of the sphere are $(\tau,2)$ and $(\tau^{-1},2)$ and where
all of the $n$ reflections of $D_{2n}$ are laid out along the equator. The upper and lower hemispheres form the two $n$-gon cells in the corresponding cell structure.
The stabilizer of this component under the conjugation action is the whole group $D_{2n}$ and the corresponding sub quotient $Q$ is the quotient
of $D_{2n}$ by its center which is $D_{2m}$ where $m=n$ if $n$ is odd and $m=\frac{n}{2}$ if $n$ is even.

Picturing this sphere as the canonical sphere of radius $1$ centered about the origin in $\mathbb{R}^3$, it is easy to compute that conjugation by $\tau$ yields a
rotation by $\frac{2\pi}{m}$ radians about the axis going thru the poles. Conjugation by a reflection $\sigma$ yields a rotation by $\pi$ radians about an axis
going through the equator of the sphere.  The resulting action is a faithful, orientation preserving action of $D_{2m}$ on the sphere which exhibits
$D_{2m}, m \geq 2$ as a subgroup of $SO(3)$ and shows that all dihedral and cyclic groups (which occur as the rotation subgroups) have strong symmetric genus zero.
\end{example}

In the next example, the explicit construction involved in the functor $\hat{Y}$ would construct the necessary example and action explicitly but we will bypass explicit analysis and argue indirectly to just show existence.

\begin{example}[Tetrahedron]
Looking at table~\ref{table:Alt4} which lists the components in $\hat{Y}(A_4)$ we find two components of genus $0$ and Schl\"afli symbol $\{3,3\}$
which are hence isomorphic as cell complexes to the tetrahedron. $A_4$ acts on these two components by conjugation and the stabilizer of one of them
is a subgroup of index $\leq 2$ i.e., $A_4$ or a subgroup of order $6$. As $A_4$ has no subgroups of order $6$, $A_4$ must be the component stabilizer.
As the tetrahedron has $4$ faces, the quotient $Q$ corresponding to that component, which acts transitively on these faces must have $4 | |Q|$ which forces $Q=A_4$ as $A_4$ has no
quotients of order exactly $4$. Thus $Q=A_4$ is contained in the orientation preserving automorphism group of the tetrahedron. On the other hand, the automorphism group of the tetrahedron
embeds into $\Sigma_4$ by considering the induced action on vertices of the tetrahedron, the orientation preserving automorphism group is hence a subgroup of
$A_4$. Putting these two faces together, we see that $A_4$ is the group of orientation preserving automorphisms of the tetrahedron
and this action occurs in the construction $\hat{Y}$.

Using a regular tetrahedron centered at the origin, one hence finds $A_4$ displayed as a subgroup of the rotation group $SO(3)$.
\end{example}

Before the next examples, we record a simple lemma about the construction $\hat{Y}(G)$ for simple groups $G$.

\begin{lem}
\label{lem:simple}
Let $G$ be a nonabelian simple group and let $S$ be a component of $\hat{Y}(G)$ whose type (genus and cell structure) occurs $k$ times in
$\hat{Y}(G)$ where $k! < |G|$. Then the sub quotient $Q$ associated to the component $S$ is $G$. If the component is non-equivar then $G=Q$ is the
group of orientation preserving automorphisms of the cell-structure of $S$.
\end{lem}
\begin{proof}
$G_S$ is an index $\leq k$ subgroup of $G$. Considering the action of $G$ on the left cosets of $G_S$, we either have $G_S=G$ or the kernel
of the action is a proper normal subgroup of $G$ and hence is trivial. Thus in the latter case, $G$ embeds in $\Sigma_k$ which is a contradiction as $k! < |G|$. Thus $G=G_S$. Again by simplicity and the fact that $Q \neq 1$, we now see that $Q=G$. The rest follows from Theorem~\ref{thm: ram}.
\end{proof}

\begin{example}[Octahedron and Cube]
Looking at table~\ref{table:Sigma4}, we find a unique component of $\hat{Y}(\Sigma_4)$ of genus $0$ and Schl\"afli symbol $\{3,4\}$ which is isomorphic as
a cell complex to the octahedron. As the component is unique and has $12$ edges the corresponding sub quotient $Q$ must equal $\Sigma_4$ as
$12 | |Q|$ and $\Sigma_4$ has no normal subgroups of order $2$. Thus $\Sigma_4$ embeds in the group of orientation preserving automorphisms
of the octahedron. However the octahedron has $12$ edges and $|\Sigma_4|=2|E|$, thus $\Sigma_4$ is the
group of orientation preserving automorphisms of the octahedron and this action occurs in the construction $\hat{Y}$.

Using a regular octahedron centered at the origin, one hence finds $\Sigma_4$ displayed as a subgroup of $SO(3)$.

Similarly from the table, there is a unique component in $\hat{Y}(\Sigma_4)$ of genus $0$ and Schl\"afli symbol $\{4,3\}$ which is isomorphic to the cube.
Analogous arguments show that this component exhibits $\Sigma_4$ as the group of orientation preserving symmetries of the cube (which of course also follows as the cube and octahedron are dual complexes).

\end{example}

\begin{example}[Dodecahedron and Icosahedron]
Looking at table~\ref{table:Alt5}, we find two components of $\hat{Y}(A_5)$ of genus $0$ and Schl\"afli symbol $\{5,3\}$ which are isomorphic to
dodecahedra. As $A_5$ is simple, the corresponding sub quotient is $Q=A_5$ by lemma~\ref{lem:simple}. As $|A_5|=60=2|E|$ we conclude that $A_5$ is the group of orientation preserving symmetries of the dodecahedron
and this action occurs in the construction $\hat{Y}$.

As usual, using a regular dodecahedron centered at the origin displays $A_5$ as a subgroup of $SO(3)$.

The table also has two components isomorphic to icosahedra and analogous arguments show $A_5$ is the group of orientation
preserving symmetries of the icosahedron and that this action arises in $\hat{Y}(A_5)$.

\end{example}

The preceding examples display cyclic groups, dihedral groups $D_{2m} (m \geq 2)$, $A_4, \Sigma_4$ and $A_5$ as finite groups with strong symmetric genus $0$.
It is known that this is the complete list, in particular all finite subgroups of $SO(3)$ are isomorphic to one of these groups.

\subsection{Extraspecial groups}

Let $p$ be an odd prime and $G$ be the extra special group of order $p^3$ and exponent $p$. We saw in Section 3.3  of \cite{HPY} that all the components in $\hat{Y}(G)$ are isomorphic and are regular, of genus $g=\frac{p(p-3)}{2}+1$, Schl\"afli symbol $\{2p,p\}$ and $p^2$ edges.
Let $S$ be one of these components, it follows from the previous analysis that the component stabilizer under the $G$ conjugation action is all of $G$ itself and the kernel of the resulting
action is $Z$, the center of $G$. Thus the resulting sub quotient $\bar{Q}=G/Z \cong C_p \times C_p$ acts
faithfully on $S$ where $C_p$ denotes the cyclic group of order $p$. This is not too interesting as  the strong symmetric genus of the elementary abelian group $\bar{Q}$ is actually $1$, thus aside from the case $p=3$ this does not display an example of an action achieving the strong symmetric genus.

If we use the action of $Aut(G)$ on $S$ instead, it is not hard to show that the resulting sub quotient $Q$ acting faithfully via orientation
preserving homeomorphisms on $S$ is $Q = (C_p \times C_p) \rtimes C_2$ where the action of $C_2$ on $C_p \times C_p$ in the semi direct product
is by coordinate exchange. As $|Q|=2p^2=2|E|$, $Q$ is the group  of orientation preserving automorphisms of the corresponding regular cell-structure on $S$.
It follows that the strong symmetric genus of $(C_p \times C_p) \rtimes C_2 \cong C_p \times D_{2p}$ is no greater than $\frac{p(p-3)}{2}+1$ for all odd primes $p$.
By comparison with the lists in \cite{FTX} we find that this is the strong symmetric genus for $p=3,5,7$. We do not know if one also has equality for
$p \geq 11$.

\subsection{Miscellaneous examples}

By table~\ref{table:PSL2F7}, the simple group $PSL(2,\mathbb{F}_7)$ which is isomorphic to $PSL(3,\mathbb{F}_2)$ has two components in $\hat{Y}$
which are surfaces of genus $3$, Schl\"afli symbol $\{14,2\text{-}3\}=D\{7,3\}$ and $168$ edges.
By simplicity, and lemma~\ref{lem:simple}, $PSL(2,\mathbb{F}_7)$ must equal the corresponding sub quotient $Q$.
Also as the cell-structure is not equivar we see that
$PSL(2,\mathbb{F}_7)$ is the group of orientation preserving automorphisms of the cell structure. It also follows that it has strong symmetric
genus no more than $3$. A quick consultation of the lists in \cite{FTX} shows that the strong symmetric genus of $PSL(2,\mathbb{F}_7)$ is indeed $3$
and so the construction $\hat{Y}$ constructs such an action. Furthermore, as $84(g-1)=168=|PSL(2,\mathbb{F}_7)|$ this is an example which achieves
the Hurwitz bound. Indeed the Riemann surface given by the Klein quartic is a Hurwitz surface of genus $3$ with automorphism group
$PSL(2,\mathbb{F}_7)$ and is the smallest genus Hurwitz surface. Similar arguments also work for the genus $3$ components with cell-structure
$\{6, 2\text{-}7\}=D\{3,7\}$ and $\{4,3\text{-}7\}$ in the same table exhibiting $PSL(2,\mathbb{F}_7)$ as the orientation preserving automorphism group of at least three non-isomorphic cell structures on the surface of genus $3$.

In table~\ref{table:Alt7}, the simple group $A_7$ has $4$ components of genus $136$, Schl\"afli symbol $\{14,2\text{-}4\}=D\{7,4\}$ and $2520$ edges.
Under the conjugation action, by simplicity and lemma~\ref{lem:simple}, $A_7$
is equal to the associated sub quotient $\bar{Q}$ for one of these components. Furthermore as we are in the non-equivar case, we see that $A_7$ is the
group of orientation preserving automorphisms of the corresponding cell structure on the surface of genus $136$. In particular the strong symmetric
genus of $A_7$ is no more than $136$. Consulting the work of Conder (\cite{C}) who determined the strong symmetric genus of symmetric and alternating groups, we see that $136$ is indeed the strong symmetric genus of $A_7$. Similar arguments hold also for the cell structures with Schl\"afli symbol
$\{8,2\text{-}7\}=D\{4,7\}$ and $\{4,4\text{-}7\}$ in the same table.

\subsection{Mapping Class Groups}

If $G$ is a nonabelian finite group and $S$ a component of $\hat{Y}(G)$ of genus $g \geq 2$ then the sub quotient $Q$ associated to that component
acts faithfully, preserving orientation on $S$. Thus under a corresponding smoothing of the action, $Q$ embeds in the orientation preserving diffeomorphism
group of the closed surface of genus $g$, i.e., $Q \leq Diff^+(X_g)$. Recall that one has the short exact sequence of topological groups:
$$
1 \to Diff_0^+(X_g) \to Diff^+(X_g) \overset{\pi}{\to} Mod_g \to 1
$$
where $Diff_0^+(X_g)$ is the connected component of the identity in $Diff^+(X_g)$ and $Mod_g$, is the mapping class group of the closed surface of
genus $g$. $Mod_g$ is a discrete, finitely generated (and hence countable) group of importance in many areas of geometry.

It is known by a result of A. Borel that for any closed aspherical manifold $M$ whose fundamental group is centerless that $Diff_0^+(M)$ is torsion free, see for example \cite{FS}. Thus in particular
$Diff_0^+(X_g)$ is torsion free for $g \geq 2$ (but definitely not for $g=0,1$). Thus the quotient homomorphism $\pi$ yields an embedding of
$Q$ into the mapping class group $Mod_g$ for $g \geq 2$.

Furthermore considering the action of $Mod_g$ on $H_1(X_g)$ one finds a short exact sequence
$$
1 \to Tor_g \to Mod_g \to Symp(2g,Z) \to 1
$$
where $Tor_g$, the Torelli subgroup, consists of elements of the mapping class group $Mod_g$ that induce the identity map on homology while
$Symp(2g,Z)$ denotes the integral symplectic group i.e. the subgroup of $GL_{2g}(\mathbb{Z})$ that preserves the symplectic (intersection/cup product) form on $\mathbb{Z}^{2g}=H_1(X_g)$.

As the Torelli subgroup is also torsion free for $g \geq 2$, we find that $Q$ also embeds into $Symp(2g,Z)$.

Note that $Q$ acts on $H^*(X_g)$ trivially on $H^2$ and $H^0$ but via the embedding into $Symp(2g,Z)$ on $H^1$. Let us denote this $Q$-module by $M$.
Thus $M$ is a free abelian group of rank $2g$ on which $Q$ acts faithfully and via symplectic automorphisms. We will call $M$ the fundamental module of the sub quotient $Q$ associated to the component $S$.

A quick spectral sequence computation with the Serre spectral sequence for the circle bundle
$S^1 \to U(X_g) \to X_g$ shows that the unit tangent bundle has integral cohomology (when $g \geq 2$) given by
\begin{eqnarray*}
H^*(U(X_g)) &\cong&  \mathbb{Z} \text{ if } * =0, 3 \\
&\cong& \mathbb{Z}^{2g} \oplus \mathbb{Z}/(2g-2)\mathbb{Z} \text{ if } * = 2 \\
&\cong& \mathbb{Z}^{2g} \text{ if } * =1
\end{eqnarray*}

The associated free $Q$-action on $U(X_g)$ is trivial on $H^0, H^3$ and the torsion of $H^2$ and is equal to $M$ on $H^1$ and $H^2$ modulo torsion.
There are many cohomological consequences (for example arising from the Browder exponent theorem (\cite{Br}) applied to the $Q$-action on $U(X_g)$) but due to space limitations we will not pursue these here.

\bigskip

\noindent
Dept. of Mathematics \\
University of Rochester, \\
Rochester, NY 14627 U.S.A. \\
E-mail address: jonpak@math.rochester.edu \\

\bigskip

\noindent
Dept. of Mathematics \\
University of Rochester, \\
Rochester, NY 14627 U.S.A. \\
E-mail address: herman@math.rochester.edu \\

\newpage

\section{Appendix: Data for Cell Complexes of Several Groups}
\label{sec: appendix}

All data in the following tables was generated using SAGE.


\begin{table}[htp]
\centering
\caption{$SL(2,\mathbb{F}_3)$ - Cell Complexes of 21 Total Components}
\begin{tabular}{|c|c|c|c|c|c|}
\hline
 Genus & \# & Schl\"afli & \# & \# & \# Components of
\\
&Faces&Symbol&Vertices&Edges& This Type
\\
\hline
0 & 2 & $\{4,2\}$ & 4 & 4 & 3
\\
(19 components) & 4 & $\{3,3\}$ & 4 & 6 & 4
\\
 & 4 & $\{6,2\text{-}3\}$ & 10 & 12 & 8
\\
 & 6 & $\{4,3\}$ & 8 & 12 & 4
 \\
\hline
1 & 4 & $\{6,3\}$ & 8 & 12 & 2
\\
(2 components) &&&&&
\\
\hline
\end{tabular}
\label{table:SL2F3}
\end{table}

%
%
%

\vskip 1cm

\begin{table}[htp]
\centering
\caption{$SL(2,\mathbb{F}_5)$ - Cell Complexes of 341 Total Components}
\begin{tabular}{|c|c|c|c|c|c|}
\hline
 Genus & \# & Schl\"afli & \# & \# & \# Components of
\\
&Faces&Symbol&Vertices&Edges& This Type
\\
\hline
0 & 2 & $\{3,2\}$ & 3 & 3 & 20
\\
(281 components) & 2 & $\{4,2\}$ & 4 & 4 & 15
\\
& 2 & $\{5,2\}$ & 5 & 5 & 24
\\
& 2 & $\{6,2\}$ & 6 & 6 & 10
\\
& 2 & $\{10,2\}$ & 10 & 10 & 12
\\
& 3 & $\{4,2\text{-}3\}$ & 5 & 6 & 40
\\
& 4 & $\{3,3\}$ & 4 & 6 & 20
\\
 & 4 & $\{6,2\text{-}3\}$ & 10 & 12 & 40
\\
& 5 & $\{4,2\text{-}5\}$ & 7 & 10 & 48
\\
 & 6 & $\{4,3\}$ & 8 & 12 & 20
 \\
& 12 & $\{5,3\}$ & 20 & 30 & 4
 \\
& 12 & $\{10,2\text{-}3\}$ & 50 & 60 & 8
 \\
& 20 & $\{3,5\}$ & 12 & 30 & 4
 \\
& 20 & $\{6,2\text{-}5\}$ & 42 & 60 & 8
 \\
& 30 & $\{4,3\text{-}5\}$ & 32 & 60 & 8
 \\
\hline
1 & 4 & $\{6,3\}$ & 8 & 12 & 10
\\
(10 components) &&&&&
\\
\hline
4 & 12 & $\{5,5\}$ & 12 & 30 & 4
\\
(16 components) & 12 & $\{10,2\text{-}5\}$ & 42 & 60 & 8
\\
& 30 & $\{4,5\}$ & 24 & 60 & 4
\\
\hline
5 & 12 & $\{10,3\}$ & 40 & 60 & 2
\\
(10 components) & 20 & $\{6,3\text{-}5\}$ & 32 & 60 & 8
\\
\hline
9 & 12 & $\{10,3\text{-}5\}$ & 32 & 60 & 16
\\
(22 components) & 20 & $\{6,5\}$ & 24 & 60 & 6
\\
\hline
13 & 12 & $\{10,5\}$ & 24 & 60 & 2
\\
(2 components) &&&&&
\\
\hline
\end{tabular}
\label{table:SL2F5}
\end{table}

\begin{table}[htp]
\centering
\caption{$SL(2,\mathbb{F}_7)$ - Cell Complexes of 1376 Total Components}
\begin{tabular}{|c|c|c|c|c|c|}
\hline
 Genus & \# & Schl\"afli & \# & \# & \# Components of
\\
&Faces&Symbol&Vertices&Edges& This Type
\\
\hline
0 & 2 & $\{3,2\}$ & 3 & 3 & 56
\\
(784 components) & 2 & $\{4,2\}$ & 4 & 4 & 42
\\
& 2 & $\{6,2\}$ & 6 & 6 & 28
\\
& 2 & $\{8,2\}$ & 8 & 8 & 42
\\
& 3 & $\{4,2\text{-}3\}$ & 5 & 6 & 112
\\
& 4 & $\{3,3\}$ & 4 & 6 & 56
\\
 & 4 & $\{4,2\text{-}4\}$ & 6 & 8 & 84
\\
 & 4 & $\{6,2\text{-}3\}$ & 10 & 12 & 112
\\
 & 6 & $\{4,3\}$ & 8 & 12 & 56
\\
 & 6 & $\{8,2\text{-}3\}$ & 20 & 24 & 56
\\
 & 8 & $\{3,4\}$ & 6 & 12 & 28
\\
 & 8 & $\{6,2\text{-}4\}$ & 18 & 24 & 56
 \\
& 12 & $\{4,3\text{-}4\}$ & 14 & 24 & 56
 \\
\hline
1 & 4 & $\{6,3\}$ & 8 & 12 & 28
\\
(92 components) & 7 & $\{6,3\}$ & 14 & 21 & 64
\\
\hline
3 & 3 & $\{14,3\}$ & 14 & 21 & 64
\\
(294 components) & 6 & $\{8,3\text{-}4\}$ & 14 & 24 & 56
\\
& 7 & $\{6,3\text{-}7\}$ & 10 & 21 & 128
\\
& 8 & $\{6,4\}$ & 12 & 24 & 14
\\
& 24 & $\{7,3\}$ & 56 & 84 & 4
\\
& 24 & $\{14,2\text{-}3\}$ & 140 & 168 & 8
\\
& 56 & $\{3,7\}$ & 24 & 84 & 4
\\
& 56 & $\{6,2\text{-}7\}$ & 108 & 168 & 8
\\
& 84 & $\{4,3\text{-}7\}$ & 80 & 168 & 8
\\
\hline
8 & 42 & $\{8,3\}$ & 112 & 168 & 8
\\
(24 components) & 56 & $\{6,3\text{-}4\}$ & 98 & 168 & 16
\\
\hline
10 & 24 & $\{7,4\}$ & 42 & 84 & 4
\\
(32 components) & 24 & $\{14,2\text{-}4\}$ & 126 & 168 & 8
\\
& 42 & $\{4,7\}$ & 24 & 84 & 4
\\
& 42 & $\{8,2\text{-}7\}$ & 108 & 168 & 8
\\
& 84 & $\{4,4\text{-}7\}$ & 66 & 168 & 8
\\
\hline
15 & 42 & $\{8,3\text{-}4\}$ & 98 & 168 & 8
\\
(12 components) & 56 & $\{6,4\}$ & 84 & 168 & 4
\\
\hline
17 & 24 & $\{14,3\}$ & 112 & 168 & 2
\\
(10 components) & 56 & $\{6,3\text{-}7\}$ & 80 & 168 & 8
\\
\hline
19 & 24 & $\{7,7\}$ & 24 & 84 & 4
\\
(14 components) & 24 & $\{14,2\text{-}7\}$ & 108 & 168 & 8
\\
& 84 & $\{4,7\}$ & 48 & 168 & 2
\\
\hline
22 & 42 & $\{8,4\}$ & 84 & 168 & 8
\\
(8 components) &&&&&
 \\
\hline
 &&&&&continued on next page
 \\
\hline
\end{tabular}
\label{table:SL2F7}
\end{table}

\clearpage

\begin{table}[htp]
\centering
\caption*{Table \ref{table:SL2F7}$: SL(2,\mathbb{F}_7)$ - continued}
\begin{tabular}{|c|c|c|c|c|c|}
\hline
 Genus & \# & Schl\"afli & \# & \# & \# Components of
\\
&Faces&Symbol&Vertices&Edges& This Type
\\
\hline
24 & 24 & $\{14,3\text{-}4\}$ & 98 & 168 & 16
\\
(48 components) & 42 & $\{8,3\text{-}7\}$ & 80 & 168 & 16
\\
& 56 & $\{6,4\text{-}7\}$ & 66 & 168 & 16
\\
\hline
31 & 24 & $\{14,4\}$ & 84 & 168 & 2
\\
(10 components) & 42 & $\{8,4\text{-}7\}$ & 66 & 168 & 8
\\
\hline
33 & 24 & $\{14,3\text{-}7\}$ & 80 & 168 & 16
\\
(22 components) & 56 & $\{6,7\}$ & 48 & 168 & 6
\\
\hline
40 & 24 & $\{14,4\text{-}7\}$ & 66 & 168 & 16
\\
(24 components) & 42 & $\{8,7\}$ & 48 & 168 & 8
\\
\hline
49 & 24 & $\{14,7\}$ & 48 & 168 & 2
\\
(2 components) &&&&&
\\
\hline
\end{tabular}
\end{table}

\clearpage

\begin{table}[htp]
\centering
\caption{$PSL(2,\mathbb{F}_7)$ - Cell Complexes of 385 Total Components}
\begin{tabular}{|c|c|c|c|c|c|}
\hline
 Genus & \# & Schl\"afli & \# & \# & \# Components of
\\
&Faces&Symbol&Vertices&Edges& This Type
\\
\hline
0 & 2 & $\{3,2\}$ & 3 & 3 & 28
\\
(245 components) & 2 & $\{4,2\}$ & 4 & 4 & 63
\\
& 3 & $\{4,2\text{-}3\}$ & 5 & 6 & 28
\\
& 4 & $\{3,3\}$ & 4 & 6 & 28
\\
 & 4 & $\{6,2\text{-}3\}$ & 10 & 12 & 28
\\
 & 6 & $\{4,3\}$ & 8 & 12 & 14
\\
 & 6 & $\{8,2\text{-}3\}$ & 20 & 24 & 14
\\
 & 8 & $\{3,4\}$ & 6 & 12 & 14
\\
 & 8 & $\{6,2\text{-}4\}$ & 18 & 24 & 14
 \\
& 12 & $\{4,3\text{-}4\}$ & 14 & 24 & 14
 \\
\hline
1 & 7 & $\{6,3\}$ & 14 & 21 & 16
\\
(16 components) &&&&&
\\
\hline
3 & 3 & $\{14,3\}$ & 14 & 21 & 16
\\
(72 components) & 6 & $\{8,3\text{-}4\}$ & 14 & 24 & 14
\\
& 7 & $\{6,3\text{-}7\}$ & 10 & 21 & 32
\\
& 24 & $\{7,3\}$ & 56 & 84 & 2
\\
& 24 & $\{14,2\text{-}3\}$ & 140 & 168 & 2
\\
& 56 & $\{3,7\}$ & 24 & 84 & 2
\\
& 56 & $\{6,2\text{-}7\}$ & 108 & 168 & 2
\\
& 84 & $\{4,3\text{-}7\}$ & 80 & 168 & 2
\\
\hline
8 & 42 & $\{8,3\}$ & 112 & 168 & 2
\\
(6 components) & 56 & $\{6,3\text{-}4\}$ & 98 & 168 & 4
\\
\hline
10 & 24 & $\{7,4\}$ & 42 & 84 & 2
\\
(10 components) & 24 & $\{14,2\text{-}4\}$ & 126 & 168 & 2
\\
& 42 & $\{4,7\}$ & 24 & 84 & 2
\\
& 42 & $\{8,2\text{-}7\}$ & 108 & 168 & 2
\\
& 84 & $\{4,4\text{-}7\}$ & 66 & 168 & 2
\\
\hline
15 & 42 & $\{8,3\text{-}4\}$ & 98 & 168 & 2
\\
(3 components) & 56 & $\{6,4\}$ & 84 & 168 & 1
\\
\hline
17 & 56 & $\{6,3\text{-}7\}$ & 80 & 168 & 2
\\
(2 components) &&&&&
\\
\hline
19 & 24 & $\{7,7\}$ & 24 & 84 & 2
\\
(4 components) & 24 & $\{14,2\text{-}7\}$ & 108 & 168 & 2
\\
\hline
22 & 42 & $\{8,4\}$ & 84 & 168 & 2
\\
(2 components) &&&&&
\\
\hline
24 & 24 & $\{14,3\text{-}4\}$ & 98 & 168 & 4
\\
(12 components) & 42 & $\{8,3\text{-}7\}$ & 80 & 168 & 4
\\
& 56 & $\{6,4\text{-}7\}$ & 66 & 168 & 4
 \\
\hline
 &&&&&continued on next page
 \\
\hline
\end{tabular}
\label{table:PSL2F7}
\end{table}

\clearpage

\begin{table}[htp]
\centering
\caption*{Table \ref{table:PSL2F7}$: PSL(2,\mathbb{F}_7)$ - continued}
\begin{tabular}{|c|c|c|c|c|c|}
\hline
 Genus & \# & Schl\"afli & \# & \# & \# Components of
\\
&Faces&Symbol&Vertices&Edges& This Type
\\
\hline
31 & 42 & $\{8,4\text{-}7\}$ & 66 & 168 & 2
\\
(2 components) &&&&&
\\
\hline
33 & 24 & $\{14,3\text{-}7\}$ & 80 & 168 & 4
\\
(5 components) & 56 & $\{6,7\}$ & 48 & 168 & 1
\\
\hline
40 & 24 & $\{14,4\text{-}7\}$ & 66 & 168 & 4
\\
(6 components) & 42 & $\{8,7\}$ & 48 & 168 & 2
\\
\hline
\end{tabular}
\end{table}

\clearpage

\begin{table}[htp]
\centering
\caption{$\Sigma_3$ - Cell Complexes of 2 Total Components}
\begin{tabular}{|c|c|c|c|c|c|}
\hline
 Genus & \# & Schl\"afli & \# & \# & \# Components of
\\
&Faces&Symbol&Vertices&Edges& This Type
\\
\hline
0 & 2 & $\{3, 2\}$ & 3 & 3 & 1
\\
0 & 3 & $\{4, 2\text{-}3\}$ & 5 & 6 & 1
\\
\hline
\end{tabular}
\label{table:Sigma3}
\end{table}

\vskip 1cm

\begin{table}[htp]
\centering
\caption{$A_4 \cong PSL(2,\mathbb{F}_3)$ - Cell Complexes of 5 Total Components}
\begin{tabular}{|c|c|c|c|c|c|}
\hline
 Genus & \# & Schl\"afli & \# & \# & \# Components of
\\
&Faces&Symbol&Vertices&Edges& This Type
\\
\hline
0 & 4 & $\{3,3\}$ & 4 & 6 & 2
\\
 & 4 & $\{6,2\text{-}3\}$ & 10 & 12 & 2
\\
 & 6 & $\{4,3\}$ & 8 & 12 & 1
\\
\hline
\end{tabular}
\label{table:Alt4}
\end{table}

\vskip 1cm

\begin{table}[htp]
\centering
\caption{$\Sigma_4$ - Cell Complexes of 27 Total Components}
\begin{tabular}{|c|c|c|c|c|c|}
\hline
 Genus & \# & Schl\"afli & \# & \# & \# Components of
\\
&Faces&Symbol&Vertices&Edges& This Type
\\
\hline
0 & 2 & $\{3,2\}$ & 3 & 3 & 4
\\
(26 components) & 2 & $\{4,2\}$ & 4 & 4 & 9
\\
 & 3 & $\{4,2\text{-}3\}$ & 5 & 6 & 4
\\
 & 4 & $\{3,3\}$ & 4 & 6 & 2
\\
 & 4 & $\{6,2\text{-}3\}$ & 10 & 12 & 2
\\
 & 6 & $\{4,3\}$ & 8 & 12 & 1
\\
 & 6 & $\{8,2\text{-}3\}$ & 20 & 24 & 1
\\
 & 8 & $\{3,4\}$ & 6 & 12 & 1
\\
 & 8 & $\{6,2\text{-}4\}$ & 18 & 24 & 1
\\
 & 12 & $\{4,3\text{-}4\}$ & 14 & 24 & 1
\\
\hline
3 & 6 & $\{8,3\text{-}4\}$ & 14 & 24 & 1
\\
\hline
\end{tabular}
\label{table:Sigma4}
\end{table}

\clearpage

\begin{table}[htp]
\centering
\caption{$A_5 \cong PSL(2,\mathbb{F}_5)$ - Cell Complexes of 91 Total Components}
\begin{tabular}{|c|c|c|c|c|c|}
\hline
 Genus & \# & Schl\"afli & \# & \# & \# Components of
\\
&Faces&Symbol&Vertices&Edges& This Type
\\
\hline
0 & 2 & $\{3,2\}$ & 3 & 3 & 10
\\
(79 components) & 2 & $\{5,2\}$ & 5 & 5 & 12
\\ & 3 & $\{4,2\text{-}3\}$ & 5 & 6 & 10
\\
 & 4 & $\{3,3\}$ & 4 & 6 & 10
\\
 & 4 & $\{6,2\text{-}3\}$ & 10 & 12 & 10
\\
 & 5 & $\{4,2\text{-}5\}$ & 7 & 10 & 12
\\
 & 6 & $\{4,3\}$ & 8 & 12 & 5
\\
 & 12 & $\{5,3\}$ & 20 & 30 & 2
\\
 & 12 & $\{10,2\text{-}3\}$ & 50 & 60 & 2
\\
 & 20 & $\{3,5\}$ & 12 & 30 & 2
\\
 & 20 & $\{6,2\text{-}5\}$ & 42 & 60 & 2
\\
 & 30 & $\{4,3\text{-}5\}$ & 32 & 60 & 2
\\
\hline
4 & 12 & $\{5,5\}$ & 12 & 30 & 2
\\
(5 components) & 12 & $\{10,2\text{-}5\}$ & 42 & 60 & 2
\\
 & 30 & $\{4,5\}$ & 24 & 60 & 1
\\
\hline
5 & 20 & $\{6,3\text{-}5\}$ & 32 & 60 & 2
\\
(2 components) &&&&&
\\
\hline
9 & 12 & $\{10,3\text{-}5\}$ & 32 & 60 & 4
\\
(5 components) & 20 & $\{6,5\}$ & 24 & 60 & 1
 \\
\hline
\end{tabular}
\label{table:Alt5}
\end{table}

\clearpage

\begin{table}[htp]
\centering
\caption{$\Sigma_5$ - Cell Complexes of 284 Total Components}
\begin{tabular}{|c|c|c|c|c|c|}
\hline
 Genus & \# & Schl\"afli & \# & \# & \# Components of
\\
&Faces&Symbol&Vertices&Edges& This Type
\\
\hline
0 & 2 & $\{3,2\}$ & 3 & 3 & 20
\\
(194 components) & 2 & $\{4,2\}$ & 4 & 4 & 45
\\
 & 2 & $\{5,2\}$ & 5 & 5 & 12
\\
 & 2 & $\{6,2\}$ & 6 & 6 & 10
\\
 & 3 & $\{4,2\text{-}3\}$ & 5 & 6 & 40
\\
 & 4 & $\{3,3\}$ & 4 & 6 & 10
\\
 & 4 & $\{6,2\text{-}3\}$ & 10 & 12 & 10
\\
 & 5 & $\{4,2\text{-}5\}$ & 7 & 10 & 12
\\
 & 6 & $\{4,3\}$ & 8 & 12 & 5
\\
 & 6 & $\{8,2\text{-}3\}$ & 20 & 24 & 5
\\
 & 8 & $\{3,4\}$ & 6 & 12 & 5
\\
 & 8 & $\{6,2\text{-}4\}$ & 18 & 24 & 5
\\
 & 12 & $\{4,3\text{-}4\}$ & 14 & 24 & 5
\\
 & 12 & $\{5,3\}$ & 20 & 30 & 2
\\
 & 12 & $\{10,2\text{-}3\}$ & 50 & 60 & 2
\\
 & 20 & $\{3,5\}$ & 12 & 30 & 2
\\
 & 20 & $\{6,2\text{-}5\}$ & 42 & 60 & 2
\\
 & 30 & $\{4,3\text{-}5\}$ & 32 & 60 & 2
\\
\hline
1 & 5 & $\{4,4\}$ & 5 & 10 & 12
\\
(24 components) & 5 & $\{8,2\text{-}4\}$ & 15 & 20 & 12
\\
\hline
3 & 6 & $\{8,3\text{-}4\}$ & 14 & 24 & 5
\\
(5 components) &&&&&
\\
\hline
4 & 4 & $\{10,4\}$ & 10 & 20 & 6
\\
(27 components) & 5 & $\{8,4\text{-}5\}$ & 9 & 20 & 12
\\
 & 12 & $\{5,5\}$ & 12 & 30 & 2
\\
 & 12 & $\{10,2\text{-}5\}$ & 42 & 60 & 2
\\
 & 24 & $\{5,4\}$ & 30 & 60 & 1
\\
 & 24 & $\{10,2\text{-}4\}$ & 90 & 120 & 1
\\
 & 30 & $\{4,5\}$ & 24 & 60 & 1
\\
 & 30 & $\{8,2\text{-}5\}$ & 84 & 120 & 1
\\
 & 60 & $\{4,4\text{-}5\}$ & 54 & 120 & 1
\\
\hline
5 & 20 & $\{6,3\text{-}5\}$ & 32 & 60 & 2
\\
(2 components) &&&&&
\\
\hline
6 & 20 & $\{6,4\}$ & 30 & 60 & 1
\\
(5 components) & 20 & $\{12,2\text{-}4\}$ & 90 & 120 & 1
\\
 & 30 & $\{4,6\}$ & 20 & 60 & 1
\\
 & 30 & $\{8,2\text{-}6\}$ & 80 & 120 & 1
\\
 & 60 & $\{4,4\text{-}6\}$ & 50 & 120 & 1
 \\
\hline
 &&&&&continued on next page
 \\
\hline
\end{tabular}
\label{table:Sigma5}
\end{table}

\clearpage

\begin{table}[htp]
\centering
\caption*{Table \ref{table:Sigma5}$: \Sigma_5$ - continued}
\begin{tabular}{|c|c|c|c|c|c|}
\hline
 Genus & \# & Schl\"afli & \# & \# & \# Components of
\\
&Faces&Symbol&Vertices&Edges& This Type
\\
\hline
9 & 12 & $\{10,3\text{-}5\}$ & 32 & 60 & 4
\\
(9 components) & 20 & $\{6,5\}$ & 24 & 60 & 1
\\
 & 20 & $\{12,2\text{-}5\}$ & 84 & 120 & 1
\\
 & 24 & $\{5,6\}$ & 20 & 60 & 1
\\
 & 24 & $\{10,2\text{-}6\}$ & 80 & 120 & 1
\\
 & 60 & $\{4,5\text{-}6\}$ & 44 & 120 & 1
 \\
\hline
11 & 20 & $\{6,6\}$ & 20 & 60 & 1
\\
(3 components) & 20 & $\{12,2\text{-}6\}$ & 80 & 120 & 1
\\
 & 30 & $\{8,3\text{-}4\}$ & 70 & 120 & 1
 \\
\hline
16 & 20 & $\{12,3\text{-}4\}$ & 70 & 120 & 2
\\
(6 components) & 30 & $\{8,3\text{-}6\}$ & 60 & 120 & 2
\\
 & 40 & $\{6,4\text{-}6\}$ & 50 & 120 & 2
 \\
\hline
19 & 30 & $\{8,4\text{-}5\}$ & 54 & 120 & 1
 \\
\hline
21 & 20 & $\{12,3\text{-}6\}$ & 60 & 120 & 1
 \\
\hline
24 & 20 & $\{12,4\text{-}5\}$ & 54 & 120 & 2
\\
(6 components) & 24 & $\{10,4\text{-}6\}$ & 50 & 120 & 2
\\
 & 30 & $\{8,5\text{-}6\}$ & 44 & 120 & 2
\\
\hline
29 & 20 & $\{12,5\text{-}6\}$ & 44 & 120 & 1
 \\
\hline
\end{tabular}
\end{table}

\clearpage

\begin{table}[htp]
\centering
\caption{$A_6$ - Cell Complexes of 1335 Total Components}
\begin{tabular}{|c|c|c|c|c|c|}
\hline
 Genus & \# & Schl\"afli & \# & \# & \# Components of
\\
&Faces&Symbol&Vertices&Edges& This Type
\\
\hline
0 & 2 & $\{3,2\}$ & 3 & 3 & 120
\\
(909 components) & 2 & $\{4,2\}$ & 4 & 4 & 135
\\
 & 2 & $\{5,2\}$ & 5 & 5 & 72
\\
 & 3 & $\{4,2\text{-}3\}$ & 5 & 6 & 120
\\
 & 4 & $\{3,3\}$ & 4 & 6 & 60
\\
 & 4 & $\{6,2\text{-}3\}$ & 10 & 12 & 60
\\
 & 5 & $\{4,2\text{-}5\}$ & 7 & 10 & 72
\\
 & 6 & $\{4,3\}$ & 8 & 12 & 30
\\
 & 6 & $\{8,2\text{-}3\}$ & 20 & 24 & 30
\\
 & 8 & $\{3,4\}$ & 6 & 12 & 30
\\
 & 8 & $\{6,2\text{-}4\}$ & 18 & 24 & 30
\\
 & 12 & $\{4,3\text{-}4\}$ & 14 & 24 & 30
\\
 & 12 & $\{5,3\}$ & 20 & 30 & 24
\\
 & 12 & $\{10,2\text{-}3\}$ & 50 & 60 & 24
\\
 & 20 & $\{3,5\}$ & 12 & 30 & 24
\\
 & 20 & $\{6,2\text{-}5\}$ & 42 & 60 & 24
\\
 & 30 & $\{4,3\text{-}5\}$ & 32 & 60 & 24
\\
\hline
1 & 9 & $\{4,4\}$ & 9 & 18 & 40
\\
(80 components) & 9 & $\{8,2\text{-}4\}$ & 27 & 36 & 40
\\
\hline
3 & 6 & $\{8,3\text{-}4\}$ & 14 & 24 & 30
\\
(30 components) &&&&&
\\
\hline
4 & 9 & $\{8,3\text{-}4\}$ & 21 & 36 & 40
\\
(120 components) & 12 & $\{5,5\}$ & 12 & 30 & 24
\\
 & 12 & $\{6,4\}$ & 18 & 36 & 20
\\
 & 12 & $\{10,2\text{-}5\}$ & 42 & 60 & 24
\\
 & 30 & $\{4,5\}$ & 24 & 60 & 12
\\
\hline
5 & 20 & $\{6,3\text{-}5\}$ & 32 & 60 & 24
\\
(24 components) &&&&&
\\
\hline
9 & 12 & $\{10,3\text{-}5\}$ & 32 & 60 & 48
\\
(60 components) & 20 & $\{6,5\}$ & 24 & 60 & 12
\\
\hline
10 & 72 & $\{5,4\}$ & 90 & 180 & 4
\\
(20 components) & 72 & $\{10,2\text{-}4\}$ & 270 & 360 & 4
\\
 & 90 & $\{4,5\}$ & 72 & 180 & 4
\\
 & 90 & $\{8,2\text{-}5\}$ & 252 & 360 & 4
\\
 & 180 & $\{4,4\text{-}5\}$ & 162 & 360 & 4
\\
\hline
16 & 90 & $\{8,3\}$ & 240 & 360 & 2
\\
(6 components) & 120 & $\{6,3\text{-}4\}$ & 210 & 360 & 4
 \\
\hline
19 & 72 & $\{5,5\}$ & 72 & 180 & 4
 \\
(8 components) & 72 & $\{10,2\text{-}5\}$ & 252 & 360 & 4
 \\
\hline
 &&&&&continued on next page
 \\
\hline
\end{tabular}
\label{table:Alt6}
\end{table}

\clearpage

\begin{table}[htp]
\centering
\caption*{Table \ref{table:Alt6}$: A_6$ - continued}
\begin{tabular}{|c|c|c|c|c|c|}
\hline
 Genus & \# & Schl\"afli & \# & \# & \# Components of
\\
&Faces&Symbol&Vertices&Edges& This Type
\\
\hline
25 & 72 & $\{10,3\}$ & 240 & 360 & 2
\\
(6 components) & 120 & $\{6,3\text{-}5\}$ & 192 & 360 & 4
\\
\hline
40 & 72 & $\{10,3\text{-}4\}$ & 210 & 360 & 8
\\
(24 components) & 90 & $\{8,3\text{-}5\}$ & 192 & 360 & 8
\\
 & 120 & $\{6,4\text{-}5\}$ & 162 & 360 & 8
\\
\hline
46 & 90 & $\{8,4\}$ & 180 & 360 & 2
 \\
(2 components) &&&&&
\\
\hline
49 & 72 & $\{10,3\text{-}5\}$ & 192 & 360 & 4
\\
(6 components) & 120 & $\{6,5\}$ & 144 & 360 & 2
\\
\hline
55 & 72 & $\{10,4\}$ & 180 & 360 & 2
\\
(10 components) & 90 & $\{8,4\text{-}5\}$ & 162 & 360 & 8
\\
\hline
64 & 72 & $\{10,4\text{-}5\}$ & 162 & 360 & 16
\\
(24 components) & 90 & $\{8,5\}$ & 144 & 360 & 8
\\
\hline
73 & 72 & $\{10,5\}$ & 144 & 360 & 6
 \\
(6 components) &&&&&
\\
\hline
\end{tabular}
\end{table}

\clearpage

\begin{table}[htp]
\centering
\caption{$\Sigma_6$ - Cell Complexes of 4477 Total Components}
\begin{tabular}{|c|c|c|c|c|c|}
\hline
 Genus & \# & Schl\"afli & \# & \# & \# Components of
\\
&Faces&Symbol&Vertices&Edges& This Type
\\
\hline
0 & 2 & $\{3,2\}$ & 3 & 3 & 240
\\
(2904 components) & 2 & $\{4,2\}$ & 4 & 4 & 540
\\
 & 2 & $\{5,2\}$ & 5 & 5 & 72
\\
 & 2 & $\{6,2\}$ & 6 & 6 & 240
\\
 & 3 & $\{4,2\text{-}3\}$ & 5 & 6 & 720
\\
 & 4 & $\{3,3\}$ & 4 & 6 & 120
\\
 & 4 & $\{6,2\text{-}3\}$ & 10 & 12 & 240
\\
 & 5 & $\{4,2\text{-}5\}$ & 7 & 10 & 72
\\
 & 6 & $\{4,3\}$ & 8 & 12 & 120
\\
 & 6 & $\{8,2\text{-}3\}$ & 20 & 24 & 120
\\
 & 8 & $\{3,4\}$ & 6 & 12 & 60
\\
 & 8 & $\{6,2\text{-}4\}$ & 18 & 24 & 120
\\
 & 12 & $\{4,3\text{-}4\}$ & 14 & 24 & 120
\\
 & 12 & $\{5,3\}$ & 20 & 30 & 24
\\
 & 12 & $\{10,2\text{-}3\}$ & 50 & 60 & 24
\\
 & 20 & $\{3,5\}$ & 12 & 30 & 24
\\
 & 20 & $\{6,2\text{-}5\}$ & 42 & 60 & 24
\\
 & 30 & $\{4,3\text{-}5\}$ & 32 & 60 & 24
\\
\hline
1 & 4 & $\{6,3\}$ & 8 & 12 & 60
\\
(284 components) & 5 & $\{4,4\}$ & 5 & 10 & 72
\\
 & 5 & $\{8,2\text{-}4\}$ & 15 & 20 & 72
\\
 & 9 & $\{4,4\}$ & 9 & 18 & 40
\\
 & 9 & $\{8,2\text{-}4\}$ & 27 & 36 & 40
\\
\hline
3 & 6 & $\{8,3\text{-}4\}$ & 14 & 24 & 120
\\
(150 components) & 8 & $\{6,4\}$ & 12 & 24 & 30
\\
\hline
4 & 4 & $\{10,4\}$ & 10 & 20 & 36
\\
(396 components) & 5 & $\{8,4\text{-}5\}$ & 9 & 20 & 72
\\
 & 6 & $\{12,2\text{-}6\}$ & 24 & 36 & 40
\\
 & 9 & $\{8,3\text{-}4\}$ & 21 & 36 & 40
\\
 & 12 & $\{5,5\}$ & 12 & 30 & 24
\\
 & 12 & $\{6,4\}$ & 18 & 36 & 20
\\
 & 12 & $\{10,2\text{-}5\}$ & 42 & 60 & 24
\\
 & 12 & $\{12,2\text{-}4\}$ & 54 & 72 & 20
\\
 & 18 & $\{4,6\}$ & 12 & 36 & 20
\\
 & 18 & $\{8,2\text{-}6\}$ & 48 & 72 & 20
\\
 & 24 & $\{5,4\}$ & 30 & 60 & 12
\\
 & 24 & $\{10,2\text{-}4\}$ & 90 & 120 & 12
\\
 & 30 & $\{4,5\}$ & 24 & 60 & 12
\\
 & 30 & $\{8,2\text{-}5\}$ & 84 & 120 & 12
\\
 & 36 & $\{4,4\text{-}6\}$ & 30 & 72 & 20
\\
 & 60 & $\{4,4\text{-}5\}$ & 54 & 120 & 12
\\
\hline
 &&&&&continued on next page
 \\
\hline
\end{tabular}
\label{table:Sigma6}
\end{table}

\begin{table}[htp]
\centering
\caption*{Table \ref{table:Sigma6}$: \Sigma_6$ - continued}
\begin{tabular}{|c|c|c|c|c|c|}
\hline
 Genus & \# & Schl\"afli & \# & \# & \# Components of
\\
&Faces&Symbol&Vertices&Edges& This Type
\\
\hline
5 & 20 & $\{6,3\text{-}5\}$ & 32 & 60 & 24
\\
(24 components) &&&&&
\\
\hline
6 & 20 & $\{6,4\}$ & 30 & 60 & 12
\\
(60 components) & 20 & $\{12,2\text{-}4\}$ & 90 & 120 & 12
\\
 & 30 & $\{4,6\}$ & 20 & 60 & 12
\\
 & 30 & $\{8,2\text{-}6\}$ & 80 & 120 & 12
\\
 & 60 & $\{4,4\text{-}6\}$ & 50 & 120 & 12
 \\
\hline
9 & 12 & $\{10,3\text{-}5\}$ & 32 & 60 & 48
\\
(108 components) & 20 & $\{6,5\}$ & 24 & 60 & 12
\\
 & 20 & $\{12,2\text{-}5\}$ & 84 & 120 & 12
\\
 & 24 & $\{5,6\}$ & 20 & 60 & 12
\\
 & 24 & $\{10,2\text{-}6\}$ & 80 & 120 & 12
\\
 & 60 & $\{4,5\text{-}6\}$ & 44 & 120 & 12
 \\
\hline
10 & 72 & $\{5,4\}$ & 90 & 180 & 4
\\
(20 components) & 72 & $\{10,2\text{-}4\}$ & 270 & 360 & 4
\\
 & 90 & $\{4,5\}$ & 72 & 180 & 4
\\
 & 90 & $\{8,2\text{-}5\}$ & 252 & 360 & 4
\\
 & 180 & $\{4,4\text{-}5\}$ & 162 & 360 & 4
\\
\hline
11 & 20 & $\{6,6\}$ & 20 & 60 & 12
\\
(36 components) & 20 & $\{12,2\text{-}6\}$ & 80 & 120 & 12
\\
 & 30 & $\{8,3\text{-}4\}$ & 70 & 120 & 12
 \\
\hline
16 & 12 & $\{12,4\text{-}6\}$ & 30 & 72 & 40
\\
(138 components) & 18 & $\{8,6\}$ & 24 & 72 & 20
\\
 & 20 & $\{12,3\text{-}4\}$ & 70 & 120 & 24
\\
 & 30 & $\{8,3\text{-}6\}$ & 60 & 120 & 24
\\
 & 40 & $\{6,4\text{-}6\}$ & 50 & 120 & 24
\\
 & 90 & $\{8,3\}$ & 240 & 360 & 2
 \\
 & 120 & $\{6,3\text{-}4\}$ & 210 & 360 & 4
 \\
\hline
19 & 30 & $\{8,4\text{-}5\}$ & 54 & 120 & 12
 \\
(20 components) & 72 & $\{5,5\}$ & 72 & 180 & 4
 \\
 & 72 & $\{10,2\text{-}5\}$ & 252 & 360 & 4
 \\
\hline
21 & 20 & $\{12,3\text{-}6\}$ & 60 & 120 & 12
 \\
(12 components) &&&&&
 \\
\hline
24 & 20 & $\{12,4\text{-}5\}$ & 54 & 120 & 24
\\
(72 components) & 24 & $\{10,4\text{-}6\}$ & 50 & 120 & 24
\\
 & 30 & $\{8,5\text{-}6\}$ & 44 & 120 & 24
\\
\hline
25 & 72 & $\{10,3\}$ & 240 & 360 & 2
\\
(6 components) & 120 & $\{6,3\text{-}5\}$ & 192 & 360 & 4
\\
\hline
29 & 20 & $\{12,5\text{-}6\}$ & 44 & 120 & 12
 \\
(12 components) &&&&&
\\
\hline
 &&&&&continued on next page
 \\
\hline
\end{tabular}
\end{table}

\clearpage

\begin{table}[htp]
\centering
\caption*{Table \ref{table:Sigma6}$: \Sigma_6$ - continued}
\begin{tabular}{|c|c|c|c|c|c|}
\hline
 Genus & \# & Schl\"afli & \# & \# & \# Components of
\\
&Faces&Symbol&Vertices&Edges& This Type
\\
\hline
40 & 72 & $\{10,3\text{-}4\}$ & 210 & 360 & 8
\\
(24 components) & 90 & $\{8,3\text{-}5\}$ & 192 & 360 & 8
\\
 & 120 & $\{6,4\text{-}5\}$ & 162 & 360 & 8
\\
\hline
46 & 90 & $\{8,4\}$ & 180 & 360 & 2
 \\
(2 components) &&&&&
\\
\hline
49 & 72 & $\{10,3\text{-}5\}$ & 192 & 360 & 4
\\
(14 components) & 120 & $\{6,5\}$ & 144 & 360 & 2
\\
 & 120 & $\{12,2\text{-}5\}$ & 504 & 720 & 2
\\
 & 144 & $\{5,6\}$ & 120 & 360 & 2
\\
 & 144 & $\{10,2\text{-}6\}$ & 480 & 720 & 2
\\
 & 360 & $\{4,5\text{-}6\}$ & 264 & 720 & 2
\\
\hline
55 & 72 & $\{10,4\}$ & 180 & 360 & 2
\\
(10 components) & 90 & $\{8,4\text{-}5\}$ & 162 & 360 & 8
\\
\hline
61 & 120 & $\{6,6\}$ & 120 & 360 & 6
\\
(15 components) & 120 & $\{12,2\text{-}6\}$ & 480 & 720 & 6
\\
& 360 & $\{4,6\}$ & 240 & 720 & 3
\\
\hline
64 & 72 & $\{10,4\text{-}5\}$ & 162 & 360 & 16
\\
(24 components) & 90 & $\{8,5\}$ & 144 & 360 & 8
\\
\hline
73 & 72 & $\{10,5\}$ & 144 & 360 & 6
 \\
(6 components) &&&&&
\\
\hline
91 & 120 & $\{12,3\text{-}4\}$ & 420 & 720 & 2
\\
(9 components) & 180 & $\{8,4\}$ & 360 & 720 & 3
\\
& 180 & $\{8,3\text{-}6\}$ & 360 & 720 & 2
\\
& 240 & $\{6,4\text{-}6\}$ & 300 & 720 & 2
\\
\hline
121 & 120 & $\{12,4\}$ & 360 & 720 & 8
\\
(42 components) & 120 & $\{12,3\text{-}6\}$ & 360 & 720 & 14
\\
& 180 & $\{8,4\text{-}6\}$ & 300 & 720 & 16
\\
& 240 & $\{6,6\}$ & 240 & 720 & 4
\\
\hline
139 & 120 & $\{12,4\text{-}5\}$ & 324 & 720 & 8
\\
(24 components) & 144 & $\{10,4\text{-}6\}$ & 300 & 720 & 8
\\
& 180 & $\{8,5\text{-}6\}$ & 264 & 720 & 8
\\
\hline
151 & 120 & $\{12,4\text{-}6\}$ & 300 & 720 & 14
\\
(21 components) & 180 & $\{8,6\}$ & 240 & 720 & 7
\\
\hline
169 & 120 & $\{12,5\text{-}6\}$ & 264 & 720 & 30
\\
(44 components) & 144 & $\{10,6\}$ & 240 & 720 & 14
\\
\hline
\end{tabular}
\end{table}

\begin{table}[htp]
\centering
\caption{$A_7$ - Cell Complexes of 16813 Total Components}
\begin{tabular}{|c|c|c|c|c|c|}
\hline
 Genus & \# & Schl\"afli & \# & \# & \# Components of
\\
&Faces&Symbol&Vertices&Edges& This Type
\\
\hline
0 & 2 & $\{3,2\}$ & 3 & 3 & 840
\\
(8379 components) & 2 & $\{4,2\}$ & 4 & 4 & 945
\\
 & 2 & $\{5,2\}$ & 5 & 5 & 252
\\
 & 2 & $\{6,2\}$ & 6 & 6 & 210
\\
 & 2 & $\{12,2\}$ & 12 & 12 & 210
\\
 & 3 & $\{4,2\text{-}3\}$ & 5 & 6 & 1260
\\
 & 4 & $\{3,3\}$ & 4 & 6 & 420
\\
 & 4 & $\{6,2\text{-}3\}$ & 10 & 12 & 840
\\
 & 5 & $\{4,2\text{-}5\}$ & 7 & 10 & 252
\\
 & 6 & $\{4,2\text{-}6\}$ & 8 & 12 & 420
\\
 & 6 & $\{4,3\}$ & 8 & 12 & 420
\\
 & 6 & $\{8,2\text{-}3\}$ & 20 & 24 & 420
\\
 & 8 & $\{3,4\}$ & 6 & 12 & 420
\\
 & 8 & $\{6,2\text{-}4\}$ & 18 & 24 & 420
\\
 & 12 & $\{4,3\text{-}4\}$ & 14 & 24 & 420
\\
 & 12 & $\{5,3\}$ & 20 & 30 & 126
\\
 & 12 & $\{10,2\text{-}3\}$ & 50 & 60 & 126
\\
 & 20 & $\{3,5\}$ & 12 & 30 & 126
\\
 & 20 & $\{6,2\text{-}5\}$ & 42 & 60 & 126
\\
 & 30 & $\{4,3\text{-}5\}$ & 32 & 60 & 126
\\
\hline
1 & 4 & $\{6,3\}$ & 8 & 12 & 210
\\
(1514 components) & 5 & $\{4,4\}$ & 5 & 10 & 252
\\
 & 5 & $\{8,2\text{-}4\}$ & 15 & 20 & 252
\\
 & 7 & $\{6,3\}$ & 14 & 21 & 240
\\
 & 9 & $\{4,4\}$ & 9 & 18 & 280
\\
 & 9 & $\{8,2\text{-}4\}$ & 27 & 36 & 280
\\
\hline
3 & 3 & $\{14,3\}$ & 14 & 21 & 240
\\
(1440 components) & 6 & $\{8,3\text{-}4\}$ & 14 & 24 & 420
\\
& 7 & $\{6,3\text{-}7\}$ & 10 & 21 & 480
\\
& 24 & $\{7,3\}$ & 56 & 84 & 60
\\
& 24 & $\{14,2\text{-}3\}$ & 140 & 168 & 60
\\
& 56 & $\{3,7\}$ & 24 & 84 & 60
\\
& 56 & $\{6,2\text{-}7\}$ & 108 & 168 & 60
\\
& 84 & $\{4,3\text{-}7\}$ & 80 & 168 & 60
\\
\hline
 &&&&&continued on next page
 \\
\hline
\end{tabular}
\label{table:Alt7}
\end{table}

\clearpage

\begin{table}[htp]
\centering
\caption*{Table \ref{table:Alt7}$: A_7$ - continued}
\begin{tabular}{|c|c|c|c|c|c|}
\hline
 Genus & \# & Schl\"afli & \# & \# & \# Components of
\\
&Faces&Symbol&Vertices&Edges& This Type
\\
\hline
4 & 4 & $\{10,4\}$ & 10 & 20 & 126
\\
(1197 components) & 5 & $\{8,4\text{-}5\}$ & 9 & 20 & 252
\\
 & 9 & $\{8,3\text{-}4\}$ & 21 & 36 & 280
\\
 & 12 & $\{5,5\}$ & 12 & 30 & 126
\\
 & 12 & $\{6,4\}$ & 18 & 36 & 140
\\
 & 12 & $\{10,2\text{-}5\}$ & 42 & 60 & 126
\\
 & 24 & $\{5,4\}$ & 30 & 60 & 21
\\
 & 24 & $\{10,2\text{-}4\}$ & 90 & 120 & 21
\\
 & 30 & $\{4,5\}$ & 24 & 60 & 63
\\
 & 30 & $\{8,2\text{-}5\}$ & 84 & 120 & 21
\\
 & 60 & $\{4,4\text{-}5\}$ & 54 & 120 & 21
\\
\hline
5 & 20 & $\{6,3\text{-}5\}$ & 32 & 60 & 126
\\
(126 components) &&&&&
\\
\hline
6 & 20 & $\{6,4\}$ & 30 & 60 & 21
\\
(105 components) & 20 & $\{12,2\text{-}4\}$ & 90 & 120 & 21
\\
 & 30 & $\{4,6\}$ & 20 & 60 & 21
\\
 & 30 & $\{8,2\text{-}6\}$ & 80 & 120 & 21
\\
 & 60 & $\{4,4\text{-}6\}$ & 50 & 120 & 21
 \\
\hline
8 & 42 & $\{8,3\}$ & 112 & 168 & 60
\\
(180 components) & 56 & $\{6,3\text{-}4\}$ & 98 & 168 & 120
 \\
\hline
9 & 12 & $\{10,3\text{-}5\}$ & 32 & 60 & 252
\\
(399 components) & 20 & $\{6,5\}$ & 24 & 60 & 63
\\
 & 20 & $\{12,2\text{-}5\}$ & 84 & 120 & 21
\\
 & 24 & $\{5,6\}$ & 20 & 60 & 21
\\
 & 24 & $\{10,2\text{-}6\}$ & 80 & 120 & 21
\\
 & 60 & $\{4,5\text{-}6\}$ & 44 & 120 & 21
 \\
\hline
10  & 24 & $\{7,4\}$ & 42 & 84 & 60
\\
(440 components) & 24 & $\{14,2\text{-}4\}$ & 126 & 168 & 60
\\
 & 42 & $\{4,7\}$ & 24 & 84 & 60
\\
 & 42 & $\{8,2\text{-}7\}$ & 108 & 168 & 60
\\
 & 72 & $\{5,4\}$ & 90 & 180 & 28
\\
 & 72 & $\{10,2\text{-}4\}$ & 270 & 360 & 28
\\
 & 84 & $\{4,4\text{-}7\}$ & 66 & 168 & 60
\\
 & 90 & $\{4,5\}$ & 72 & 180 & 28
\\
 & 90 & $\{8,2\text{-}5\}$ & 252 & 360 & 28
\\
 & 180 & $\{4,4\text{-}5\}$ & 162 & 360 & 28
\\
\hline
 &&&&&continued on next page
 \\
\hline
\end{tabular}
\end{table}

\clearpage

\begin{table}[htp]
\centering
\caption*{Table \ref{table:Alt7}$: A_7$ - continued}
\begin{tabular}{|c|c|c|c|c|c|}
\hline
 Genus & \# & Schl\"afli & \# & \# & \# Components of
\\
&Faces&Symbol&Vertices&Edges& This Type
\\
\hline
11 & 20 & $\{6,6\}$ & 20 & 60 & 21
\\
(63 components) & 20 & $\{12,2\text{-}6\}$ & 80 & 120 & 21
\\
 & 30 & $\{8,3\text{-}4\}$ & 70 & 120 & 21
 \\
\hline
15 & 42 & $\{8,3\text{-}4\}$ & 98 & 168 & 60
\\
(90 components) & 56 & $\{6,4\}$ & 84 & 168 & 30
\\
\hline
16 & 20 & $\{12,3\text{-}4\}$ & 70 & 120 & 42
\\
(168 components) & 30 & $\{8,3\text{-}6\}$ & 60 & 120 & 42
\\
 & 40 & $\{6,4\text{-}6\}$ & 50 & 120 & 42
\\
 & 90 & $\{8,3\}$ & 240 & 360 & 14
 \\
 & 120 & $\{6,3\text{-}4\}$ & 210 & 360 & 28
 \\
\hline
17 & 56 & $\{6,3\text{-}7\}$ & 80 & 168 & 60
\\
(60 components) &&&&&
\\
\hline
19 & 24 & $\{7,7\}$ & 24 & 84 & 60
 \\
(197 components) & 24 & $\{14,2\text{-}7\}$ & 108 & 168 & 60
 \\
& 30 & $\{8,4\text{-}5\}$ & 54 & 120 & 21
 \\
& 72 & $\{5,5\}$ & 72 & 180 & 28
 \\
 & 72 & $\{10,2\text{-}5\}$ & 252 & 360 & 28
 \\
\hline
21 & 20 & $\{12,3\text{-}6\}$ & 60 & 120 & 21
 \\
(21 components) &&&&&
 \\
\hline
22 & 42 & $\{8,4\}$ & 84 & 168 & 60
 \\
(60 components) &&&&&
 \\
\hline
24 & 20 & $\{12,4\text{-}5\}$ & 54 & 120 & 42
\\
(486 components) & 24 & $\{10,4\text{-}6\}$ & 50 & 120 & 42
\\
& 24 & $\{14,3\text{-}4\}$ & 98 & 168 & 120
\\
 & 30 & $\{8,5\text{-}6\}$ & 44 & 120 & 42
\\
& 42 & $\{8,3\text{-}7\}$ & 80 & 168 & 120
\\
& 56 & $\{6,4\text{-}7\}$ & 66 & 168 & 120
\\
\hline
25 & 72 & $\{10,3\}$ & 240 & 360 & 14
\\
(42 components) & 120 & $\{6,3\text{-}5\}$ & 192 & 360 & 28
\\
\hline
29 & 20 & $\{12,5\text{-}6\}$ & 44 & 120 & 21
 \\
(21 components) &&&&&
\\
\hline
31 & 42 & $\{8,4\text{-}7\}$ & 66 & 168 & 60
 \\
(60 components) &&&&&
\\
\hline
33 & 24 & $\{14,3\text{-}7\}$ & 80 & 168 & 120
 \\
(150 components) & 56 & $\{6,7\}$ & 48 & 168 & 30
\\
\hline
40 & 24 & $\{14,4\text{-}7\}$ & 66 & 168 & 120
\\
(348 components) & 42 & $\{8,7\}$ & 48 & 168 & 60
\\
& 72 & $\{10,3\text{-}4\}$ & 210 & 360 & 56
\\
& 90 & $\{8,3\text{-}5\}$ & 192 & 360 & 56
\\
 & 120 & $\{6,4\text{-}5\}$ & 162 & 360 & 56
\\
\hline
 &&&&&continued on next page
\\
\hline
\end{tabular}
\end{table}

\clearpage

\begin{table}[htp]
\caption*{Table \ref{table:Alt7}$: A_7$ - continued}
\begin{tabular}{|c|c|c|c|c|c|}
\hline
 Genus & \# & Schl\"afli & \# & \# & \# Components of
\\
&Faces&Symbol&Vertices&Edges& This Type
\\
\hline
46 & 90 & $\{8,4\}$ & 180 & 360 & 14
 \\
(14 components) &&&&&
\\
\hline
49 & 72 & $\{10,3\text{-}5\}$ & 192 & 360 & 28
\\
(42 components) & 120 & $\{6,5\}$ & 144 & 360 & 14
\\
\hline
55 & 72 & $\{10,4\}$ & 180 & 360 & 14
\\
(70 components) & 90 & $\{8,4\text{-}5\}$ & 162 & 360 & 56
\\
\hline
64 & 72 & $\{10,4\text{-}5\}$ & 162 & 360 & 112
\\
(168 components) & 90 & $\{8,5\}$ & 144 & 360 & 56
\\
\hline
73 & 72 & $\{10,5\}$ & 144 & 360 & 42
 \\
(42 components) &&&&&
\\
\hline
136 & 360 & $\{7,4\}$ & 630 & 1260 & 4
\\
(20 components) & 360 & $\{14,2\text{-}4\}$ & 1890 & 2520 & 4
\\
& 630 & $\{4,7\}$ & 360 & 1260 & 4
\\
& 630 & $\{8,2\text{-}7\}$ & 1620 & 2520 & 4
\\
& 1260 & $\{4,4\text{-}7\}$ & 990 & 2520 & 4
\\
\hline
169 & 504 & $\{10,3\}$ & 1680 & 2520 & 1
\\
(3 components) & 840 & $\{6,3\text{-}5\}$ & 1344 & 2520 & 2
\\
\hline
199 & 360 & $\{7,5\}$ & 504 & 1260 & 4
\\
(20 components) & 360 & $\{14,2\text{-}5\}$ & 1764 & 2520 & 4
\\
& 504 & $\{5,7\}$ & 360 & 1260 & 4
\\
& 504 & $\{10,2\text{-}7\}$ & 1620 & 2520 & 4
\\
& 1260 & $\{4,5\text{-}7\}$ & 864 & 2520 & 4
\\
\hline
211 & 420 & $\{12,3\}$ & 1680 & 2520 & 1
\\
(15 components) & 630 & $\{8,3\text{-}4\}$ & 1470 & 2520 & 8
\\
& 840 & $\{6,4\}$ & 1260 & 2520 & 4
\\
& 840 & $\{6,6\text{-}3\}$ & 1260 & 2520 & 2
\\
\hline
241 & 360 & $\{7,6\}$ & 420 & 1260 & 4
\\
(26 components) & 360 & $\{14,3\}$ & 1680 & 2520 & 2
\\
& 360 & $\{14,2\text{-}6\}$ & 1680 & 2520 & 4
\\
& 420 & $\{6,7\}$ & 360 & 1260 & 4
\\
& 420 & $\{12,2\text{-}7\}$ & 1620 & 2520 & 4
\\
& 840 & $\{6,3\text{-}7\}$ & 1200 & 2520 & 4
\\
& 1260 & $\{4,6\text{-}7\}$ & 780 & 2520 & 4
\\
\hline
271 & 360 & $\{7,7\}$ & 360 & 1260 & 6
\\
(13 components) & 360 & $\{14,2\text{-}7\}$ & 1620 & 2520 & 6
\\
& 1260 & $\{4,7\}$ & 720 & 2520 & 1
\\
\hline
274 & 504 & $\{10,3\text{-}4\}$ & 1470 & 2520 & 10
\\
(30 components) & 630 & $\{8,3\text{-}5\}$ & 1344 & 2520 & 10
\\
& 840 & $\{6,4\text{-}5\}$ & 1134 & 2520 & 10
\\
\hline
 &&&&&continued on next page
 \\
\hline
\end{tabular}
\end{table}

\clearpage

\begin{table}[htp]
\centering
\caption*{Table \ref{table:Alt7}$: A_7$ - continued}
\begin{tabular}{|c|c|c|c|c|c|}
\hline
 Genus & \# & Schl\"afli & \# & \# & \# Components of
\\
&Faces&Symbol&Vertices&Edges& This Type
\\
\hline
316 & 420 & $\{12,3\text{-}4\}$ & 1470 & 2520 & 6
\\
(30 components) & 630 & $\{8,4\}$ & 1260 & 2520 & 12
\\
& 630 & $\{8,3\text{-}6\}$ & 1260 & 2520 & 6
\\
& 840 & $\{6,4\text{-}6\}$ & 1050 & 2520 & 6
\\
\hline
337 & 504 & $\{10,3\text{-}5\}$ & 1344 & 2520 & 6
\\
(9 components) & 840 & $\{6,5\}$ & 1008 & 2520 & 3
\\
\hline
346 & 360 & $\{14,3\text{-}4\}$ & 1470 & 2520 & 16
\\
(48 components) & 630 & $\{8,3\text{-}7\}$ & 1200 & 2520 & 16
\\
& 840 & $\{6,4\text{-}7\}$ & 990 & 2520 & 16
\\
\hline
379 & 420 & $\{12,3\text{-}5\}$ & 1344 & 2520 & 6
\\
(51 components) & 504 & $\{10,4\}$ & 1260 & 2520 & 11
\\
& 504 & $\{10,3\text{-}6\}$ & 1260 & 2520 & 6
\\
& 630 & $\{8,4\text{-}5\}$ & 1134 & 2520 & 22
\\
& 840 & $\{6,5\text{-}6\}$ & 924 & 2520 & 6
\\
\hline
409 & 360 & $\{14,3\text{-}5\}$ & 1344 & 2520 & 18
\\
(54 components) & 504 & $\{10,3\text{-}7\}$ & 1200 & 2520 & 18
\\
& 840 & $\{6,5\text{-}7\}$ & 864 & 2520 & 18
\\
\hline
421 & 420 & $\{12,4\}$ & 1260 & 2520 & 6
\\
(21 components) & 420 & $\{12,3\text{-}6\}$ & 1260 & 2520 & 2
\\
& 630 & $\{8,4\text{-}6\}$ & 1050 & 2520 & 12
\\
& 840 & $\{6,6\}$ & 840 & 2520 & 1
\\
\hline
442 & 504 & $\{10,4\text{-}5\}$ & 1134 & 2520 & 8
\\
(12 components) & 630 & $\{8,5\}$ & 1008 & 2520 & 4
\\
\hline
451 & 360 & $\{14,4\}$ & 1260 & 2520 & 20
\\
(88 components) & 360 & $\{14,3\text{-}6\}$ & 1260 & 2520 & 8
\\
 & 420 & $\{12,3\text{-}7\}$ & 1200 & 2520 & 8
\\
& 630 & $\{8,4\text{-}7\}$ & 990 & 2520 & 44
\\
& 840 & $\{6,6\text{-}7\}$ & 780 & 2520 & 8
\\
\hline
481 & 360 & $\{14,3\text{-}7\}$ & 1200 & 2520 & 26
\\
(37 components) & 840 & $\{6,7\}$ & 720 & 2520 & 11
\\
\hline
484 & 420 & $\{12,4\text{-}5\}$ & 1134 & 2520 & 8
\\
(24 components) & 504 & $\{10,4\text{-}6\}$ & 1050 & 2520 & 8
\\
 & 630 & $\{8,5\text{-}6\}$ & 924 & 2520 & 8
\\
\hline
505 & 504 & $\{10,5\}$ & 1008 & 2520 & 1
\\
\hline
514 & 360 & $\{14,4\text{-}5\}$ & 1134 & 2520 & 36
\\
(108 components) & 504 & $\{10,4\text{-}7\}$ & 990 & 2520 & 36
\\
 & 630 & $\{8,5\text{-}7\}$ & 864 & 2520 & 36
\\
\hline
526 & 420 & $\{12,4\text{-}6\}$ & 1050 & 2520 & 4
\\
(6 components) & 630 & $\{8,6\}$ & 840 & 2520 & 2
\\
\hline
547 & 420 & $\{12,5\}$ & 1008 & 2520 & 4
\\
(12 components) & 504 & $\{10,5\text{-}6\}$ & 924 & 2520 & 8
\\
\hline
\end{tabular}
\end{table}

\clearpage

\begin{table}[htp]
\caption*{Table \ref{table:Alt7}$: A_7$ - continued}
\begin{tabular}{|c|c|c|c|c|c|}
\hline
 Genus & \# & Schl\"afli & \# & \# & \# Components of
\\
&Faces&Symbol&Vertices&Edges& This Type
\\
\hline
556 & 360 & $\{14,4\text{-}6\}$ & 1050 & 2520 & 16
\\
(48 components) & 420 & $\{12,4\text{-}7\}$ & 990 & 2520 & 16
\\
& 630 & $\{8,6\text{-}7\}$ & 780 & 2520 & 16
\\
\hline
577 & 360 & $\{14,5\}$ & 1008 & 2520 & 10
\\
(34 components) & 504 & $\{10,5\text{-}7\}$ & 864 & 2520 & 24
\\
\hline
586 & 360 & $\{14,4\text{-}7\}$ & 990 & 2520 & 40
\\
(60 components) & 630 & $\{8,7\}$ & 720 & 2520 & 20
\\
\hline
589 & 420 & $\{12,5\text{-}6\}$ & 924 & 2520 & 2
\\
(3 components) & 504 & $\{10,6\}$ & 840 & 2520 & 1
\\
\hline
619 & 360 & $\{14,5\text{-}6\}$ & 924 & 2520 & 14
\\
(42 components) & 420 & $\{12,5\text{-}7\}$ & 864 & 2520 & 14
\\
& 504 & $\{10,6\text{-}7\}$ & 780 & 2520 & 14
\\
\hline
631 & 420 & $\{12,6\}$ & 840 & 2520 & 1
\\
\hline
649 & 360 & $\{14,5\text{-}7\}$ & 864 & 2520 & 48
\\
(70 components) & 504 & $\{10,7\}$ & 720 & 2520 & 22
\\
\hline
661 & 420 & $\{12,6\text{-}7\}$ & 780 & 2520 & 4
\\
(4 components) &&&&&
\\
\hline
691 & 360 & $\{14,6\text{-}7\}$ & 780 & 2520 & 14
\\
(21 components) & 420 & $\{12,7\}$ & 720 & 2520 & 7
\\
\hline
721 & 360 & $\{14,7\}$ & 720 & 2520 & 20
\\
(20 components) &&&&&
\\
\hline
\end{tabular}
\end{table}

\end{document}